\documentclass[12pt]{amsart}

\usepackage{standalone}
\usepackage{amssymb}
\usepackage{bbm}
\usepackage{enumitem}
\usepackage{url}
\usepackage{hyperref}
\usepackage{comment}
\usepackage{geometry}
\usepackage{mathtools}
\usepackage{ytableau}
\usepackage{bbold}
\usepackage{subcaption}
\usepackage{rotating}
\usepackage{tikz,graphicx}
\usetikzlibrary{decorations.pathreplacing}
\usetikzlibrary{arrows.meta}
\usetikzlibrary{positioning}

\tikzstyle{vertex}=[circle, draw, inner sep=0pt, minimum size=4pt]
\newcommand{\vertex}{\node[vertex]}

\newtheorem{theorem}{Theorem}[section]
\newtheorem{lemma}[theorem]{Lemma}
\newtheorem{proposition}[theorem]{Proposition}
\newtheorem{corollary}[theorem]{Corollary}
\theoremstyle{definition}
\newtheorem{definition}[theorem]{Definition}
\newtheorem{example}[theorem]{Example}

\newtheorem{remark}[theorem]{Remark}

\newcommand{\defn}[1]{\emph{\color{blue} #1}}

\newcommand\bg{\mathbf{g}}
\newcommand\bk{\mathbf{k}}

\newcommand\bx{\mathbf{x}}

\newcommand\bzero{\mathbf{0}}

\newcommand\bbR{\mathbb{R}}

\newcommand\R{\mathbb{R}}

\newcommand\calB{\mathcal{B}}

\newcommand\calF{\mathcal{F}}

\newcommand\calS{\mathcal{S}}
\newcommand\calT{\mathcal{T}}

\newcommand\conv{\mathrm{conv}}
\newcommand\pull{\mathrm{pull}}
\newcommand{\DKK}{\mathrm{DKK}}

\newcommand\In{\mathrm{In}}
\newcommand\inedge{\mathrm{in}}
\newcommand\indeg{\mathrm{indeg}}
\newcommand\Out{\mathrm{Out}}
\newcommand\outedge{\mathrm{out}}
\newcommand\outdeg{\mathrm{outdeg}}

\newcommand\Path{\mathrm{Path}}
\newcommand\Cycle{\mathrm{Cycle}}
\newcommand\north{\mathrm{N}}
\newcommand\south{\mathrm{S}}
\newcommand\east{\mathrm{E}}
\newcommand\west{\mathrm{W}}
\newcommand\NW{\mathrm{NW}}
\newcommand\NE{\mathrm{NE}}
\newcommand\SW{\mathrm{SW}}
\newcommand\SE{\mathrm{SE}}

\newcommand\squarette{\scalebox{0.75}{\hbox{$\,\square\,$}}}

\title{Locally anti-blocking $\bg$-polytopes for flow polytopes}

\author[Berggren]{Jonah Berggren}
\address{Department of Mathematics\\
         University of Kentucky
}
\email{jrberggren@uky.edu}
\urladdr{https://sites.google.com/view/jonahberggrenmath/}

\author[Braun]{Benjamin Braun}
\address{Department of Mathematics\\
         University of Kentucky
}
\email{benjamin.braun@uky.edu}
\urladdr{https://sites.google.com/view/braunmath/}

\author[Cornejo]{Alvaro Cornejo}
\address{Department of Mathematics\\
         University of Kentucky
}\email{alvaro.cornejo@uky.edu}
\urladdr{http://cornejomath.github.io/}

\author[McElroy]{James Ford McElroy}
\address{Lexington, KY
}
\urladdr{https://www.linkedin.com/in/james-ford-mcelroy/ }

\author[Napier]{Chloe' Napier}
\address{Department of Mathematics\\
         University of Kentucky
}
\email{Chloe.Napier@uky.edu}
\urladdr{https://math.as.uky.edu/users/cmma241}

\author[Peterson]{Zachery Peterson}
\address{Department of Mathematics\\
         William and Mary\\
}
\email{ ztpeterson@wm.edu}
\urladdr{https://www.wm.edu/as/mathematics/faculty-directory/peterson\_z.php}

\author[Rizer]{Williem Rizer}
\address{Department of Mathematics\\
         University of Kentucky
}
\email{williem.rizer@uky.edu}
\urladdr{https://sites.google.com/view/williemrizer/home}

\author[Serhiyenko]{Khrystyna Serhiyenko}
\address{Department of Mathematics\\
         University of Kentucky
}
\email{khrystyna.serhiyenko@uky.edu}
\urladdr{https://math.as.uky.edu/users/kse246}

\author[Yip]{Martha Yip}
\address{Department of Mathematics\\
         University of Kentucky
}
\email{mpyi222@uky.edu}
\urladdr{https://www.ms.uky.edu/~myip/}

\date{25 May 2026}

\thanks{
BB, AC, JFM, and WR were partially supported by NSF award DMS-1953785.
BB and AC were partially supported by NSF award DMS-2450299.
KS, JB, ZP, and CN were partially supported by NSF award DMS-2054255.
KS and JB were partially supported by NSF award DMS-2451909.
This work was supported by a grant from the
Simons Foundation International [SFI-MPS-TSM-00013650, KS].
MY was partially supported by Simons Collaboration grant 964456.
}

\begin{document}

\begin{abstract}
    Given an acyclic directed graph (DAG), the space of strength one flows is a lattice polytope called the flow polytope of the DAG.
    If the DAG admits an ample framing, then the flow polytope is Gorenstein and it linearly projects onto a reflexive polytope called the $\mathbf{g}$-polytope.
    We provide a combinatorial characterization of amply framed DAGs that have a locally anti-blocking $\mathbf{g}$-polytope, and we characterize the minimal faces of the $\mathbf{g}$-polytope containing a fixed pair of vertices.
    We prove in this case that the unimodular triangulation of the $\mathbf{g}$-polytope induced by the DKK triangulation of the flow polytope is a pulling triangulation, and we characterize the pulling orders that yield the DKK triangulation.
    To prove our results, we introduce and study coherence diagrams, a combinatorial model of coherence for amply framed DAGs with locally anti-blocking $\mathbf{g}$-polytopes.
    We conclude by indicating possible extensions of these results to the setting of $\mathbf{g}$-polytopes for gentle Nakayama algebras.
\end{abstract}

\maketitle

\section{Introduction}\label{sec:intro}

\subsection{Context and motivation}

Given an acyclic directed graph (DAG) $G$, the space of flows of strength one on $G$ is a lattice polytope called the \emph{flow polytope} of $G$, denoted $\calF_1(G)$.
Algebraic, geometric, and combinatorial properties of flow polytopes have been the subject of intense investigation over the past several decades, with particular emphasis on subdivisions and triangulations~\cite{BGHHKMY19,berggren2026framingtriangulationsframingposets,berggrenframing,BraunCornejo,DKK12,permutationflows1}, volume and Ehrhart series formulas~\cite{BV08,gensnakeposets,bipartiteflowextensions,braunmcelroy,cryoriginal,GHMY,kms2021,pitmanstanleygeneralized,pitmanstanleyvolumelattice,lidskiiformula,LMS19,LMM16,MM19,MMS19,moralesshi2021,pitmanstanleyoriginal}, lattices arising from dual graphs of unimodular triangulations~\cite{FramingLattices,BGMY23,GMPTY25}, toric algebra~\cite{domokos_equations_toric_2016,HILLE2003215,RietschWilliams_Flow_Toric_2025}, and connections with representation theory of gentle algebras~\cite{gfan_uqam_arxiv,Polytopes2,Berggren,berggrenturbulence,berggrenserhiyenko}.
A key tool in many of these works is a family of regular unimodular triangulations of $\calF_1(G)$ introduced by Danilov, Karzanov, and Koshevoy known as \emph{DKK triangulations}~\cite{DKK12} that are defined by the combinatorial data given by a \emph{framing} of $G$.

The class of \emph{amply framed} DAGs has played an important role in recent work due to strong connections between these objects and $\tau$-tilting theory for gentle algebras~\cite{gfan_uqam_arxiv,Polytopes2,Berggren,berggren2026framingtriangulationsframingposets,berggrenserhiyenko}.
Flow polytopes for DAGs that admit an ample framing are examples of \emph{Gorenstein} lattice polytopes, which have an extensive history and which play important roles in commutative algebra~\cite{hibidepthreflexive,specialsimplicestoricrings}, combinatorics~\cite{Athanasiadis,beckfreesum,braundavissolus,BrunsRomer,chapotonqgorenstein,rootsgorensteinfano,gorensteingraphicmatroids,gorensteinlecturehall,facetsvolumegorenstein,hibireflexiveposets,gorensteintrinomial,gorensteinmatroids,lorenznill,kohlolsensanyal,nillvolumereflexive,nillschepers,ohsugigorensteincut,tsuchiyagorensteinsimplices,tsuchiya2026classificationcountinggorensteinsimplices,tsuchiya2026gorensteinsimplicesbinaryselfcomplementary}, and algebraic geometry~\cite{batyrev,coxmirrorsurvey}.
Gorenstein flow polytopes admit several characterizations related to the combinatorics of the defining DAG $G$~\cite{Polytopes2,BraunCornejo}.
More generally, Gorenstein polytopes include the special class of \emph{reflexive} polytopes, which originally arose in the context of mirror symmetry~\cite{batyrev}.
Bruns and R\"omer proved that if a Gorenstein polytope has the integer decomposition property, which is an arithmetic property of the polytope, then it admits (at least one) projection onto a reflexive polytope~\cite{BrunsRomer}.
The resulting pair of polytopes share many important combinatorial, geometric, and algebraic properties.
Lattice polytopes with unimodular triangulations always have the integer decomposition property, and thus every Gorenstein flow polytope $\calF_1(G)$ admits such a projection.
For a particular choice of projection, the resulting reflexive polytope $P$ is known as the $\bg$-polytope for $G$, due to the fact that the face structure of $P$ is the same as that of the $\bg$-vector fan for the gentle algebra corresponding to $G$.

In the setting of combinatorial optimization, Fulkerson introduced the class of anti-blocking polytopes~\cite{fulkersonblockingantiblocking,fulkersonantiblocking}, which extend to the class of \emph{locally anti-blocking} polytopes.
Anti-blocking and locally anti-blocking polytopes have been the subject of recent study in Ehrhart theory and polyhedral geometry~\cite{sanyalgeometricantiblocking,kohlolsensanyal,ohsugitsuchiyalocallyantiblocking,RR}.
Locally anti-blocking reflexive polytopes were characterized by Kohl, Olsen, and Sanyal~\cite{kohlolsensanyal}, and this characterization involves compressed polytopes and pulling triangulations.
Given a family of reflexive polytopes, it is therefore of interest to classify those that are locally anti-blocking and to study their pulling triangulations and geometric structure.
This is the motivation for the present work.

\subsection{Our contributions}

Our goal is to study geometric and combinatorial properties of locally anti-blocking $\bg$-polytopes of amply framed DAGs.
Our main contributions are:
\begin{enumerate}
    \item Theorem~\ref{thm:lab-dag-characterization} proves that a connected DAG admits an ample framing with a locally anti-blocking $\bg$-polytope if and only if the inner graph of the DAG is a path or a cycle; we denote such DAGs by $\Path(\bk)$ and $\Cycle(\bk)$, respectively.
    \item Corollary~\ref{cor.2framings} characterizes the ample framings of $\Path(\bk)$ and $\Cycle(\bk)$ that yield a locally anti-blocking $\bg$-polytope, which we call locally anti-blocking ample framings.
    \item Corollary~\ref{cor:incoherence-square} characterizes the minimal face of the $\bg$-polytope containing the $\bg$-vectors for a pair of routes in $G$.
    \item Lemma~\ref{lem:thecoherencelemma} gives a combinatorial model encoding coherence for pairs of routes in $\Path(\bk)$ or $\Cycle(\bk)$ when equipped with a locally anti-blocking ample framing.
    \item Corollary~\ref{cor:characterizing-pull-coherent-routes} characterizes pulling sequences for $\bg$-polytopes that yield the DKK triangulation, and Theorem~\ref{thm:dkkpulling} establishes that the DKK triangulation of a $\bg$-polytope is a pulling triangulation.
\end{enumerate}

The remainder of this paper is structured as follows.
In Section~\ref{sec:background}, we provide background regarding pulling triangulations, Gorenstein and reflexive polytopes, locally anti-blocking polytopes, flow polytopes, and $\bg$-polytopes.
In Section~\ref{sec:labdags}, we introduce the DAGs $\Path(\bk)$ and $\Cycle(\bk)$ and use them to characterize the amply framed DAGs with locally antiblocking $\bg$-polytopes.
In Section~\ref{sec:pulling}, we study the minimal face of the $\bg$-polytope containing a fixed pair of vertices, introduce route pair notation for a locally anti-blocking amply framed DAG, introduce the combinatorial tool of coherence diagrams and establish their properties, and characterize the pulling triangulations of the $\bg$-polytope that agree with the DKK triangulation.
We conclude with Section~\ref{sec:furtherdirections}, where we indicate possible extensions of these results to the setting of $\bg$-polytopes for gentle Nakayama algebras.

\section{Background}\label{sec:background}

Unless specified otherwise, throughout this work we assume that all polytopes are lattice polytopes, where a \defn{lattice polytope} is a convex polytope with vertices contained in the integer lattice.

\subsection{Pulling triangulations and compressed polytopes}\label{sec:pulling-triangulations}

A pulling triangulation~\cite{triangulationshandbook} of a convex polytope $P$ is defined by recursively refining subdivisions of $P$.
At each step we select a point $x\in P$ and \defn{pull} at $x$ by subdividing each cell $S$ of the subdivision containing $x$ into the pyramids $\conv(f,x)$ for each face $f$ of $S$ that does not contain $x$.
In this section, Theorems~\ref{thm:pulling1} and~\ref{thm:iteratedpulling} depend only on this general definition of pulling and make no special considerations for points on the integer lattice.
Thus, for our initial discussion of pulling and subdivisions, we will not require $P$ to be a lattice polytope.

\begin{definition}
    Let $P$ be a $d$-dimensional polytope. A \defn{subdivision} of $P$ is a finite collection $\calS = \{S_1,\dots,S_m\}$ of $d$-dimensional polytopes such that
\begin{enumerate}
    \item $P = \cup_{i=1}^m S_i$, and
    \item for $i\neq j$, $F:=S_i \cap S_j$ is a common (possibly empty) proper face of $S_i$ and $S_j$.
\end{enumerate}     
    Let $x,y\in P$. 
    We say $y$ is \defn{cell-neighboring} to $x$ in a subdivision $\calS$ if there is some cell $S_i$ containing both of them.
    If no such cell exists, we say that $x$ and $y$ are \defn{separated} in $\calS$.
    A subdivision $\mathcal{S'}$ \defn{refines} a subdivision $\calS$ if each cell of $\calS'$ is contained in some cell of $\calS$.
    A \defn{triangulation} is a subdivision $\calT$ in which each cell is a simplex.
\end{definition}

A recursive process to produce subdivisions and triangulations is the pulling process, defined as follows.

\begin{definition}
    Suppose $\calS=\{S_1,\dots,S_m\}$ is a subdivision of $P$, with $x \in P$.
    The result of \defn{pulling $\calS$ at $x$}, denoted $\pull(\calS; x) =: \calS'$, is the refinement of $\calS$ defined by modifying each $S_i$ as follows.
    \begin{itemize}
        \item If $x\notin S_i$, then $S_i\in\calS'$.
        \item If $x\in S_i$, then for every facet $F$ of $S_i$ not containing $x$, we have $\conv(F,x)\in\calS'$.
    \end{itemize}

    For any sequence $\bx = (x^1,\dots, x^n)$ of points in a polytope $P$, we denote the result of successively pulling at each point as $\pull(P;\bx) \coloneq \pull( \dots \pull(P; x^1) \dots; x^n)$, and we call $\bx$ the \defn{pulling order}.
    If $\calS = \pull(P;\bx)$, then we denote the $i^{\text{th}}$ subdivision $\calS^i := \pull(P;(x^1,\dots,x^i))$, with $\calS^0 := \{P\}$.
    Any subdivision arising in this way is called a \defn{pulling subdivision}, or a \defn{pulling triangulation} if it is a triangulation of $P$.
\end{definition}

\begin{remark}\label{remark:face_pulling}
More explicitly the face structure of a pulling subdivision is given by the following.
Suppose $x\in S_i$ and $F$ is a facet of $S_i$ not containing $x$.
Then by the definition of pulling at $x$, we introduce a new cell $S':= \conv(F,x)$.
The faces of $S'$ are then given by $f$ and $\conv(f,x)$ for each face $f$ of $F$.
That is to say, pulling at $x$ subdivides each $k$-dimensional face $f$ containing it into the $k$-dimensional faces defined by $\conv(f',x)$ for each $(k-1)$-dimensional face $f'\subseteq f$ not containing $x$.
\end{remark}

Note that if $x$ is the first pull-point in a pulling triangulation, it will be in every simplex of the resulting pulling triangulation.
In fact, the following is an equivalent way to view pulling (as shown in the next few theorems): given a subdivision $\calS = \{S_1,\dots,S_m\}$ of $P$, consider all triangulations of $P$ that refine $\calS$.
By pulling at a point $x$, we are precisely narrowing down to those triangulations for which $x$ is in every simplex of each cell $S_i$ containing $x$.
The following theorem gives a characterization for when two points are separated in a pulling subdivision of a polytope.

\begin{theorem}\label{thm:pulling1} 
Suppose $u$ and $v$ are vertices of a subdivision $\calS$ of a polytope $P$, with $u$ and $v$ cell-neighboring, i.e., $u,v \in S\in \calS$. 
Let $\calS' = \pull(\calS;x)$ for some $x\in P$.
Then $u$ and $v$ are separated in $\calS'$ if and only if $x\neq u$, $x\neq v$, and $x$ is in the minimal face (by containment) containing both $u$ and $v$.
\end{theorem}
    \begin{proof}
    Let $f$ be the minimal face of $\calS$ containing both $u$ and $v$.   
    First, we show the forward direction by the contrapositive. 
    Note that if $x\in\{u,v\}$, then the edge $(u,v)$ exists in $\calS'$, meaning $u$ and $v$ are not separated.
    Now suppose $x\notin f$.
        If $x\notin S$, then $S$ remains in the subdivision, giving a cell containing $u$ and $v$.
        If $x\in S$, then since $x\notin f$, we have $\conv(f,x)$ is a face of some cell $S'\in \calS'$, implying that $S'$ contains both $u$ and $v$.
        Hence $u$ and $v$ are not separated.

        For the reverse implication, let $x\in f$, where $x\notin \{u,v\}$, and suppose that pulling at $x$ does not separate $u$ and $v$.
        In other words, there is some $S'\in \calS'$ that contains both $u$ and $v$.
        But considering Remark \ref{remark:face_pulling}, since $u$ and $v$ are vertices of the subdivision, there must have been some face $f'\subseteq f\subseteq S'$ that contained $u$ and $v$ but not $x$, contradicting that $f$ is minimal.
    \end{proof}

The following theorem gives a condition by which we can ensure that iterated pulling subdivisions are refined by the same triangulation.

\begin{theorem}\label{thm:iteratedpulling}
    Let $\calS=\{S_1,\dots,S_m\}$ be a subdivision of $P$ and let $\calT$ be a triangulation refining $\calS$. Then $\calT$ also refines $\calS':=\pull(\calS;v)$ for a vertex $v$ if and only if $v$ is a vertex of $\calT$ such that for every cell $S\in \calS$ containing $v$ we have $v\in T$ for every simplex $T\subseteq S$ in $\calT$.
\end{theorem}
    \begin{proof}
        We first prove the forward implication. Suppose that $\calS'$ is refined by $\calT$. Since $v\in\calS'$ by the definition of pulling, we must have $v\in\calT$ by the definition of refinement. Now let $S\in\calS$ with $v\in S$. By the definition of pulling, $\calS'$ contains the edge $(u,v)$ for each vertex $u\in S$, $u\neq v$. Since $\calT$ refines $\calS'$, $\calT$ also contains the edge $(u,v)$. Because $\calT$ contains the edge $(u,v)$ for each $u\in S$ with $u\neq v$, $v$ must lie in every simplex $T\in\calT$ with $T\subseteq S$. 
        
        We now prove the reverse implication. Consider any cell $S\in \calS$ with $v\in S$. Since $\calT$ refines $\calS$, $S$ is triangulated by
        $\calT' \coloneq \calT\vert_S =: \{T_1',\dots, T_n'\}$.
        Assuming that $v$ is in every simplex of $\calT'$, and letting $\calB$ denote the boundary complex of $S$ minus any facet containing $v$, we see that the collection $\{ B_i'\coloneq T_i'\setminus\{v\}\}_{i=1}^n$ is a triangulation of $\calB$ (where $T_i'\setminus\{v\}$ is shorthand for the convex hull of the vertex set of $T_i'$ except for $v$).
        This means that every $(d-1)$-dimensional simplex $B_i'$ is contained in some facet $f$ of $S$ with $v\notin f$.
        Hence, by definition, the $d$-dimensional simplex $T_i = \conv(B_i',x) \in \calT$ is contained in the cell $\conv(f,x)\in\calS'$.
        Applying this to every cell containing $v$, we find $\calT$ refines $\calS'$.
    \end{proof}

Returning to the setting of lattice polytopes, there are two types of pulling (weak and strong), but in the case where every lattice point in a lattice polytope is a vertex of a subdivision that is being refined by a pulling, these two variants of pulling agree~\cite{unimodulartriangulationssurvey}.
In this paper we will be applying pullings in the special case described above, and therefore we can ignore the distinction between weak and strong pulling.

For lattice polytopes where the vertices of a triangulation are contained in the lattice points in the polytope, the following class of triangulations is of particular importance.

\begin{definition}
    A \defn{unimodular triangulation} of a lattice polytope is a triangulation such that every cell is unimodular; that is, each cell is lattice equivalent to the standard simplex obtained as the convex hull of the origin and the standard basis vectors.
\end{definition}

\begin{definition}
    A lattice polytope is \defn{compressed} if all its pulling triangulations are unimodular. 
\end{definition}

The following characterization of compressed polytopes has multiple proofs, originally by Santos~\cite{unimodulartriangulationssurvey}, with subsequent proofs by Ohsugi and Hibi~\cite{ohsugihibicompressed} and Sullivant~\cite{sullivantcompressed}.

\begin{theorem}
    Let $P$ be a lattice polytope. The following are equivalent:

    \begin{enumerate}
        \item $P$ is compressed.
        \item $P$ has width one with respect to all its facets.
        \item $P$ is lattice equivalent to the intersection of a unit cube with an affine space.
    \end{enumerate}
\end{theorem}

\subsection{Gorenstein and reflexive polytopes}

The polytopes arising in this work fall within the family of Gorenstein and reflexive polytopes, defined as follows.

\begin{definition}\label{def:reflexive}
    A $d$-dimensional lattice polytope $P$ is \defn{reflexive} if it contains an interior lattice point $v$ and there exists an integer matrix $A$ such that
    \[
        P - v = \{ x \in \mathbb{R}^d: A x \leq \mathbb{1}\}.
    \]
    A lattice polytope $P$ is \defn{Gorenstein} if there exists some $r \in \mathbb{N}$ such that $r P$ is reflexive.
\end{definition}

There is a robust theory describing projections of unimodularly-triangulated Gorenstein polytopes onto reflexive polytopes with shared geometric and arithmetical properties, due to Bruns and R\"omer~\cite{BrunsRomer}.
To review this theory, we first recall the definitions of a Gorenstein point and a special simplex.
Recall that for a lattice polytope $P$, the \defn{cone over $P$} is
\[
\mathrm{cone}(P):=\mathrm{span}_{\R_\geq 0}\{(1,v): v \in P\}\, .
\]
Recall also that for a full-dimensional convex body $K\subset \R^d$, we write $K^\circ$ to denote the interior of $K$.

\begin{proposition}
    A lattice polytope $P \subseteq \mathbb{R}^{d}$ is \defn{Gorenstein} if and only if there exists a point $c \in \mathbb{Z}^{d+1}$ such that $c + \mathrm{cone}(P) \cap \mathbb{Z}^{d+1}=\mathrm{cone}(P)^\circ \cap \mathbb{Z}^{d+1}$. 
    If such a point $c$ exists, we call it the \defn{Gorenstein point} for $P$.
\end{proposition}

\begin{definition}
    Given a $d$-dimensional lattice polytope $P$, a simplex $\Delta$ with vertices in $P \cap \mathbb{Z}^d$ is \defn{special} if $\Delta \cap F$ is a facet of $\Delta$ for all facets $F$ of $P$.
\end{definition}

Theorem~\ref{thm:brunsromer} below is the key structural result underlying the properties we describe later in this work for faces and triangulations of $\bg$-polytopes in relation to flow polytopes.
Note that the integer decomposition theorem mentioned in the theorem statement is a property of a lattice polytope that is implied by having a regular unimodular triangulation~\cite{unimodulartriangulationssurvey}.
Thus, the flow polytopes considered in this paper have the integer decomposition property.

\begin{theorem}[Bruns and R\"omer~\cite{BrunsRomer}]\label{thm:brunsromer}

If $P$ is a Gorenstein polytope with the integer decomposition property and Gorenstein point $c$, then there must exist a sum 
\[
c=(1,y_1)+(1,y_2)+\cdots+(1,y_r)
\]
with $(1,y_i) \in 1\times P$ for every $i$.
In this case, the simplex $\Delta:=\mathrm{conv}\{y_1,\ldots,y_r\}$ is a special simplex for $P$. 
Further, if $P$ admits a unimodular triangulation $T$, then there exists a simplicial sphere $S$ in $T$ such that $P$ is unimodularly triangulated by the join $S \star conv\{y_1,\ldots,y_r\}$.
Further, the projection of $P$ along the linear subspace parallel to the affine span of $\Delta$ yields a reflexive polytope with a unimodular triangulation isomorphic to the join of $S$ and a point.
\end{theorem}

Recall that the \defn{join} of two polytopes $P\subset \R^{d_P}$ and $Q\subset \R^{d_Q}$ is 
\[
P\star Q := \mathrm{conv}\{P\times\{0_Q\}\times \{0\}\cup \{0_P\}\times Q \times \{1\} \}\subset \R^{d_P+d_Q+1}
\]
and that if $P$ and $Q$ both contain the zero vector in their interior, then the \defn{free sum} of $P$ and $Q$ is
\[
P\oplus Q := \mathrm{conv}\{P\times\{0_Q\} \cup \{0_P\}\times Q \}\subset \R^{d_P+d_Q} \, .
\]
If $P$ and $Q$ are both Gorenstein with special simplices $\Delta_P$ and $\Delta_Q$, respectively, then it is straightforward to verify that $P\star Q$ is Gorenstein with special simplex $\Delta_P\star \Delta_Q$.
Note also that if $P$ and $Q$ both contain the zero vector in their interior, then the free sum $P\oplus Q$ is the projection of the join $P\star Q$ along the linear subspace generated by the unit vector in the last coordinate.
The join construction for lattice polytopes can be equivalently defined as the convex hull of two polytopes in skew affine subspaces of $\R^d$ such that those skew affine subspaces are at distance one from each other with respect to a primitive integer linear functional.
For two lattice polytopes $P$ and $Q$ in this situation that each contain a designated interior lattice point, their free sum with respect to that pair of interior lattice points is the projection of $P\star Q$ along the linear subspace parallel to the line between those two interior points. 
The join and free sum constructions defined above are special cases of the construction in this general context~\cite{hrz97,m76,ps67}.

We will need the following result regarding the Bruns-R\"omer projections for joins of Gorenstein polytopes.

\begin{proposition}\label{prop:joinfreesumgorenstein}
    Assume that $P$ and $Q$ satisfy the hypotheses of Theorem~\ref{thm:brunsromer}.
    If projecting $P$ along a special simplex $\Delta_P$ yields the reflexive polytope $P'$, and if projecting $Q$ along a special simplex $\Delta_Q$ yields the reflexive polytope $Q'$, then projecting $P\star Q$ along the special simplex $\Delta_P\star \Delta_Q$ yields the reflexive polytope $P'\oplus Q'$.
    Further, the structure of the unimodular triangulations in Theorem~\ref{thm:brunsromer} for the join and free sum is induced by those of $P$ and $Q$.
\end{proposition}

\begin{proof}
The Gorenstein condition for a lattice polytope is characterized by palindromicity of the Ehrhart $h^*$-polynomial~\cite{stanleyhilbertgraded}.
The $h^*$-polynomial of a join $P\star Q$ is the product of the $h^*$-polynomials of $P$ and $Q$~\cite{BR07}, and thus $P\star Q$ is Gorenstein since products of palindromic polynomials are palindromic.
Further, if $P$ and $Q$ have the integer decomposition property, it is a straightforward exercise to show that their join does as well; for all polytopes considered in this work, $P$ and $Q$ have regular unimodular triangulations, and the join of two regular unimodular triangulations is a regular unimodular triangulation.
Thus, the join $P\star Q$ also satisfies the hypothesis of Theorem~\ref{thm:brunsromer}, and the special simplex for $P\star Q$ is the join of the special simplices $\Delta_P$ and $\Delta_Q$ for $P$ and $Q$, respectively.
Hence, the join of the triangulated spheres $S_P$ and $S_Q$ arising in the triangulations of $P$ and $Q$, respectively, given by Theorem~\ref{thm:brunsromer} is the sphere in the decomposition of $P\star Q$ arising from that theorem.
In other words, the triangulation of $P\star Q$ given by Theorem~\ref{thm:brunsromer} is $(S_P\star S_Q)\star (\Delta_P\star \Delta_Q)$.

What remains is to prove that the projection of $P\star Q$ along the linear space parallel to the affine span of the special simplex $\Delta_P\star \Delta_Q$ yields the free sum $P'\oplus Q'$.
Observe that projecting $P\star Q$ along the sum of the affine span of $\Delta_P$ and the affine span of $\Delta_Q$ yields the join $P'\star Q'$.
However, the affine span of $\Delta_P\star \Delta_Q$ has one dimension more than the sum of the dimensions of the affine span of $\Delta_P$ and the affine span of $\Delta_Q$, and this is accounted for by the subspace parallel to the line through the unique interior points of $P'$ and $Q'$ in $P'\star Q'$.
Projecting along this one-dimensional subspace completes the quotient process, and yields the free sum $P'\oplus Q'$.
\end{proof}

We will also require the following proposition regarding pulling triangulations of free sums of reflexive polytopes.

\begin{proposition}\label{prop:freesumpulling}
If $P$ and $Q$ are reflexive polytopes with pulling triangulations $T_P$ and $T_Q$ where each of those pulling triangulations begin by pulling at the origin, then pulling in $P\oplus Q$ first at the origin and then at the points of $P$ in the order for $T_P$ followed by the points of $Q$ in the order for $T_Q$ yields a pulling triangulation of $P\oplus Q$ isomorphic to the triangulation obtained by taking the cone at the origin over the topological join $T_P|_{\partial P}\star T_Q|_{\partial Q}$.
Further, if $T_P$ and $T_Q$ are unimodular, then so is the resulting triangulation of the free sum.
\end{proposition}

\begin{proof}
    This follows by observing that the facets of $P\oplus Q$ are given by the joins of facets of $P$ with facets of $Q$ and that facets of reflexive polytopes are at lattice distance one from the origin.
\end{proof}

\subsection{Locally anti-blocking polytopes}\label{sec:locallyanti-blocking}

Our main polyhedral property of interest is the concept of a locally anti-blocking polytope, for which we need some notation.
Let $P\subset \mathbb{R}^d$ be a convex lattice polytope, let $\mathbb{R}^d_+ = \{ x \in \mathbb{R}^d: x_1,\ldots, x_d\geq0\}$ denote the positive orthant, and let $P_+= P\cap \mathbb{R}^d_+$ denote the restriction of $P$ to the positive orthant.

\begin{definition}
The polytope $P_+$ is \defn{anti-blocking} if for any $x\in P_+$ and $y\in \mathbb{R}^d$ with $0\leq y_i\leq x_i$ for all $i=1,\ldots, d$, then $y\in P_+$.
\end{definition}

\begin{definition}
Let $\sigma\in \{-1,1\}^d$.
For $x\in \mathbb{R}^d$, let $\sigma x = (\sigma_1 x_1,\ldots, \sigma_d x_d)$.
A polytope $P \subset \mathbb{R}^d$ is \defn{locally anti-blocking} if $P_\sigma := (\sigma P) \cap\mathbb{R}^d_+$ is anti-blocking for every $\sigma\in \{-1,1\}^d$.
\end{definition}

Anti-blocking polyhedra arose originally in optimization, but have been found to play important roles in geometric combinatorics as well~\cite{kohlolsensanyal, RR}.
The following proposition gives a method for checking the locally anti-blocking condition.

\begin{proposition}\label{prop:lab-iff-projections} 
A polytope $P$ is locally anti-blocking if and only if for each vertex $v\in P$ we have $\pi_I(v)\in P$ for each $I\in 2^{[d]}$, where $\pi_I$ is the projection defined by $\mathbf{e}_i\mapsto \mathbf{0}$ for $i\in I$.
\end{proposition}

\begin{proof}
If $P$ is locally anti-blocking, then for any vector in $P$, an arbitrary collection of coordinates can be set to zero and the resulting point remains in $P$. 
This establishes the forward direction.
For the reverse direction, suppose that $w\in P$.
Thus, $w$ is a convex combination of vertices of $P$.
Suppose that $w'$ is obtained from $w$ by scaling the $i$-th coordinate of $w$ by a factor $0\leq r <1$.
In the convex combination defining $w$, we can scale the $i$-th coordinate of every vertex of $P$ by $r$, and the resulting combination will yield $w'$. 
Because of the vector obtained by scaling the $i$-th coordinate of a vertex $v$ by $r$ is in line segment formed by $v$ and $v$ with the $i$-th coordinate set to $0$, our assumption yields that every such scaled vertex is in $P$.
Thus, $w'$ is also in $P$.
Iterating this process to account for any number of coordinates of $w$, we find that $P$ is locally anti-blocking.
\end{proof}

A key result motivating our work is the following characterization of locally anti-blocking reflexive polytopes, which connects the locally anti-blocking condition to the compressed condition.

\begin{theorem}[Kohl, Olsen, Sanyal~\cite{kohlolsensanyal}]\label{thm:labcompressed}
If $P$ is a locally anti-blocking polytope with the origin in the interior, then $P$ is reflexive if and only if for every $\sigma\in \{1,-1\}^d$ we have that $P_\sigma$ is compressed.
\end{theorem}

The locally anti-blocking property is preserved by the free sum operation.

\begin{proposition}\label{prop:labfreesum}
    If $P$ and $Q$ are both locally anti-blocking polytopes with their respective origins in their interiors, then the free sum $P\oplus Q$ is locally anti-blocking.
\end{proposition}

\begin{proof}
Proposition~\ref{prop:lab-iff-projections} states that it is sufficient to check the locally anti-blocking condition on vertices of a polytope.
Since the vertex set of the free sum is the union of the vertices of the summands, the projection condition from Proposition~\ref{prop:lab-iff-projections} follows immediately from the projection condition being satisfied by the summands.
\end{proof}

\subsection{Flow polytopes}\label{sec:flowpolytopes}

Let $G=(V,E)$ be a finite acyclic directed graph (DAG) with vertex set $V = \{0,\ldots, n+1\}$ and edge multiset $E$.
For a vertex $v\in V$, let $\inedge(v)$ (respectively $\outedge(v)$) denote the set of edges incoming to $v$ (respectively the set of edges outgoing from $v$).
Any vertex of $G$ that is not a source vertex or a sink vertex is an \defn{inner vertex}.

\begin{definition}
A \defn{flow} $f$ on a DAG $G$ is a function $f:E \rightarrow \bbR_{\geq0}$ satisfying
\[
\sum_{e\in \inedge(v)} f(e) = \sum_{e\in \outedge(v)} f(e)
\]
for every inner vertex $v$.
The \defn{flow polytope $\calF_1(G)$} is the set of all flows on $G$ of size $1$, that is,
\[
\sum_{\genfrac{}{}{0pt}{}{\hbox{\tiny $v$ is a source}}{\hbox{\tiny $e\in \outedge(v)$}}} f(e) =1.
\]
\end{definition}

A \defn{route} of $G$ is a directed path from a source vertex to a sink vertex of $G$.  
An alternative definition of the flow polytope $\calF_1(G)$ is that it is the convex hull of indicator vectors in $\mathbb{R}^E$ of the routes of $G$.
An edge of $G$ is \defn{idle} if it is the only incoming or outgoing edge from an inner vertex.
Contracting an idle edge $e$ in $G$ yields a graph whose flow polytope is integrally equivalent to the flow polytope on $G$.
This is because the contraction of $e$ corresponds to projection along the coordinate $x_e$.
For the purposes of this article, there is no loss of generality in assuming that a DAG has no idle edges.

Danilov, Karzanov and Koshevoy~\cite{DKK12} developed a combinatorial method for constructing a collection of regular unimodular triangulations of a flow polytope.
We next describe their construction.

\begin{definition}
A \defn{framing} at an inner vertex $v$ of $G$ is a choice of linear orders $\preceq_{F,\inedge(v)}$ and $\preceq_{F,\outedge(v)}$ on the set of edges incoming to $v$ and on the set of edges outgoing from $v$.  
A \defn{framing} of a DAG $G$ is a framing at every inner vertex of $G$.
A \defn{framed graph} $(G,F)$ is a DAG $G$ with a framing $F$.
\end{definition}

Let $v$ be an inner vertex of $G$.
Let $\In(v)$ denote the set of partial routes in $G$ from a source vertex to $v$ and let $\Out(v)$ denote the set of partial routes in $G$ from $v$ to a sink vertex.
The framing of $G$ induces a total order on the sets of partial routes of $G$ at $v$ as follows.

\begin{definition}
Let $v$ be an inner vertex of $G$, and let $R, S \in \In(v)$ be partial routes which coincide on the subroutes $R'\subseteq R$ and $S'\subseteq S$ from vertex $u$ to $v$ so that the edges immediately preceding $u$ on $R$ and on $S$ are distinct, which we denote by $e_R$ and $e_S$.
Let $R \preceq_{F,\In(v)} S$ if and only if $e_R \preceq_{F, \inedge(v)} e_S$.
The definition of $\preceq_{F,\Out(v)}$ is analogous.
\end{definition}

Let $R$ be a route of $G$ that contains the inner vertex $v$.
Let $Rv$ (respectively $vR$) denote the subroute from the source of $R$ to $v$ (respectively the subroute from $v$ to the sink of $R$).
\begin{definition}
Let $R$, $S$ be routes of $G$ containing a common subroute from $u$ to $v$, denoted $[u,v]$ (where $u$ may be the same vertex as $v$).
The routes $R$ and $S$ are \defn{in conflict at $[u,v]$} if $Ru$ and $Su$ are ordered differently from $vR$ and $vS$.  
Otherwise, they are \defn{coherent at $[u,v]$}.
$R$ and $S$ are \defn{coherent} if they are coherent at every common subroute.
A route is \defn{exceptional} if it is coherent with every route of $G$.
\end{definition}

\begin{definition}
A \defn{clique} of $G$ is a collection of mutually coherent routes, with respect to the framing $F$.    
\end{definition}

\begin{theorem}[{Danilov, Karzanov, Koshevoy~\cite{DKK12}}]\label{thm:dkk-triangulation}
Let $(G,F)$ be a framed DAG.
The set of cliques of $G$ with respect to the framing $F$ forms a regular unimodular triangulation of $\calF_1(G)$.
\end{theorem}
This is the \defn{DKK triangulation of $\calF_1(G)$ corresponding to the framing $F$}, and we denote it by $\DKK(G,F)$.
In this article we are mainly concerned with DKK triangulations arising from the special class of ample framings. 

\begin{definition}\label{def:ampleframing}
A framing is \defn{ample} if the set of exceptional routes it induces is not contained in any facet of the positive cone of flows on $G$.
\end{definition}

The definition of an ample framing was motivated by the problem of determining when the reduced fan of the DKK triangulation of the cone of flows on $G$ is complete, see~\cite{DKK12, PPS} for a more in-depth discussion.
\begin{theorem}[{\cite[Proposition 5]{DKK12}}]
Let $(G,F)$ be a framed graph.  
The following are equivalent:
\begin{enumerate}
    \item $F$ is an ample framing,
    \item the quotient of the cone of nonnegative flows on $G$ by the linear span of the exceptional routes of $(G,F)$ is a complete fan,
    \item each non-idle edge of $G$ belongs to an exceptional route for $F$.
\end{enumerate}
\end{theorem}

Bell et al.~\cite[Section 3]{Polytopes2} characterized the class of DAGs that admit ample framings, which prompted the following definition. 

\begin{definition}
A DAG $G$ is \defn{full} if for each of its inner vertices $v$, we have $\indeg(v)= \outdeg(v) = 2$.
\end{definition}

The following proposition follows from a result of Bell et al.~\cite{Polytopes2}, namely that the flow polytope for a DAG $G$ is Gorenstein if and only if the in-degree and out-degree of $G$ is equal at each inner vertex of $G$.

\begin{proposition}\label{prop:fullgorenstein}
    The flow polytope for a full DAG is Gorenstein.
\end{proposition}

It was shown in~\cite[Corollary 3.14]{Polytopes2} that a DAG with no idle edges admits an ample framing if and only if $G$ is full.
Furthermore, an ample framing of $G$ induces a labeling of the edges of a full DAG $G$ by $\{1,2\}$ so that at each inner vertex, the two incoming edges have distinct labels, and the two outgoing edges have distinct labels (see~\cite[Corollary 3.11]{Polytopes2}).
Conversely, any edge labeling of $G$ by $1$ and $2$ having the property that at every inner vertex the incoming (and outgoing) edges have distinct labels induces an ample framing of $G$.
In particular, a route in $G$ is exceptional with respect to the ample framing $F$ if and only if the edges in the route have constant labels.

\begin{example}\label{ex:ampleframingpath}
Figure~\ref{fig:framing_edge_labels} shows a full DAG that we refer to as $\Path(3,4,2)$ on the vertex set $\{0,\ldots, 11\}$ (the motivation for this naming convention is explained in Section~\ref{sec:pathcycledags}).
The edge labeling of $G$ by $\{1,2\}$ depicted in the figure has the property that at every inner vertex, the pair of incoming edges (and the pair of outgoing edges) has distinct labels.  
This induces an ample framing $F$ of $G$ such that at each inner vertex of $G$, the incoming and outgoing framing orders are given by the edge labels of the incident edges. 
\end{example}

\begin{figure}
    \centering
    \includegraphics{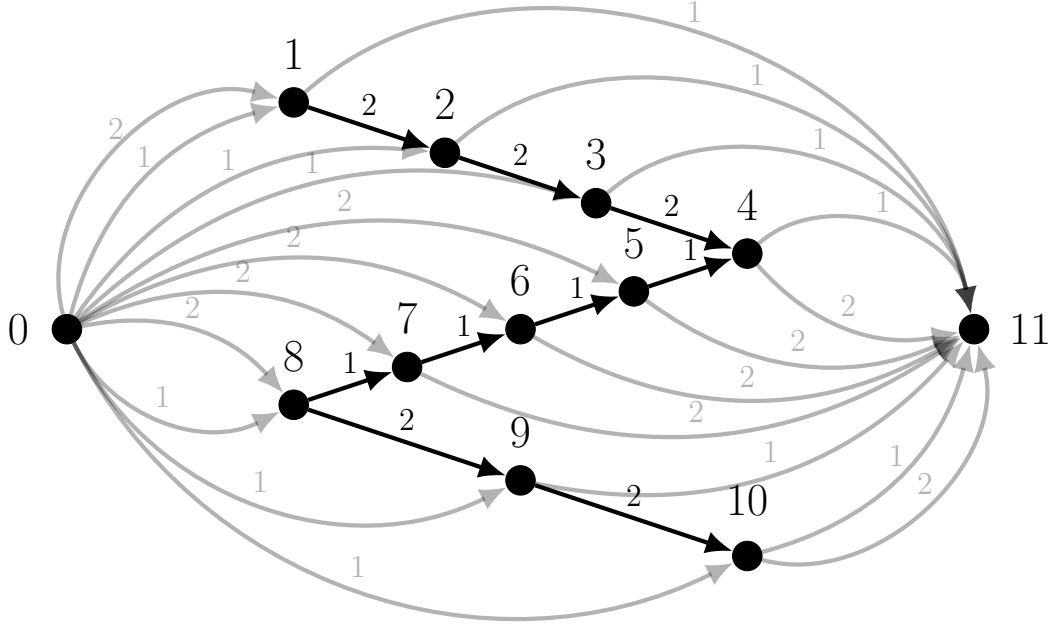}
    \caption{The edge labeling of $\Path(3,4,2)$ induced from an ample framing, as discussed in Example~\ref{ex:ampleframingpath}.}
    \label{fig:framing_edge_labels}
\end{figure}

\begin{definition}
Given a DAG $G$, the directed subgraph that is induced by the inner vertices of $G$ is its \defn{inner graph}.
An \defn{inner route} is a maximal directed path of the inner graph.
\end{definition}

\begin{example}
The inner graph of the DAG in Figure~\ref{fig:framing_edge_labels} is shown with bolded edges.
Viewed as an undirected graph, the inner graph is an undirected path from vertex $1$ to vertex $10$.
\end{example}

When the inner graph of a DAG is disconnected, the resulting flow polytope is a join.

\begin{theorem}\label{thm:dagdisconnectjoin}
    If the inner graph of a DAG $G$ is disconnected and the edge set of $G$ is the union of the edge sets of the DAGs $G_1, G_2, \ldots, G_t$ where each $G_i$ has inner graph corresponding to a connected component of the inner graph of $G$, then 
    \[
\calF_1(G) = \calF_1(G_1)\star \calF_1(G_2)\star \cdots \star \calF_1(G_t) \, .
    \]
\end{theorem}

\begin{proof}
    We have that 
    \[
    \R^{E(G)}=\R^{E(G_1)}+\R^{E(G_2)}+\cdots +\R^{E(G_t)} \, .
    \]
    Let $\lambda_i$ denote the linear functional on $\R^{E(G)}$ given by the sum of the flow values on the edges incident to the source in $G_i$.
    Each flow polytope $\calF_1(G_i)$ is contained in the affine subspace defined by 
    \[
    V_i:=\{x:\lambda_i(x)=1\}\bigcap \cap_{e\notin G_i}\{x:x_e=0\} \, .
    \]
    These are skew affine subspaces in $\R^{E(G)}$ and $V_i$ is at lattice distance one from any other $V_j$ via the form $\lambda_i$. 
    The result follows.
\end{proof}

\subsection{\texorpdfstring{$\bg$}{g}-polytopes of framed DAGs}

We next define $\bg$-vectors and $\bg$-polyhedra in the setting of framed DAGs.

\begin{definition} \label{def:gvector}
Let $R$ be a route of an amply framed $(G,F)$.
The \defn{$\bg$-vector} of $R$, denoted $\bg(R) = (a_v)_{v} \in \mathbb{Z}^n$, is defined by
\[
a_v = \begin{cases}
    -1, & \hbox{if $R$ has $\stackrel{1}{\longrightarrow} v \stackrel{2}{\longrightarrow}$,} \\
    1, & \hbox{if $R$ has $\stackrel{2}{\longrightarrow} v \stackrel{1}{\longrightarrow}$,} \\
    0, & \hbox{otherwise.}     
\end{cases}
\]
The \defn{$\bg$-polytope} of $(G,F)$ is the convex hull of the $\bg$-vectors of $(G,F)$, which we denote by $\bg(G,F)$. 
\end{definition}

As described by Berggren~\cite{Berggren} and by Braun and Cornejo~\cite{BraunCornejo}, the $\bg$-vector construction agrees with a projection map on the flow polytope of $(G,F)$ along a linear subspace parallel to the affine span of the simplex formed by the exceptional routes, hence this map is of the type given by Theorem~\ref{thm:brunsromer}.
It therefore follows from either Proposition~\ref{prop:fullgorenstein} combined with Theorem~\ref{thm:brunsromer}, or independently from Remark~\ref{rem:gentleg}, that $\bg(G,F)$ is reflexive.
Observe that $\bg(R) =\mathbf{0}$ if and only if $R$ is an exceptional route of $(G,F)$.
Any other route has a $\bg$-vector which is a vertex of the $\bg$-polyhedron, as proved in Corollary~\ref{cor:vertices}.

\begin{remark}\label{rem:gentleg}
The definitions above are special cases of definitions in the more general setting of turbulence polyhedra for gentle algebras~\cite{Berggren}.
The connection is that given an amply framed DAG $G$, we can construct a quiver related to $G$ to obtain a gentle algebra $\Lambda(G)$~\cite{Polytopes2}. 
We may then consider the $\bg$-vector of a route of $G$ to be the $\bg$-vector of the corresponding route of the blossoming algebra of $\Lambda(G)$, and we may consider the triangulated $\bg$-polytope to be a quotient of the flow polytope $\mathcal F_1(G)=\mathcal F_1(\Lambda(G))$.
It was proved in~\cite{AHIKM,Berggren} that the $\bg$-polytopes of representation-finite gentle algebras, and hence also of amply framed DAGs, are reflexive.
\end{remark}

An immediate consequence of Theorem~\ref{thm:brunsromer} is that a DKK triangulation of the flow polytope for an amply framed DAG induces a unimodular triangulation of the $\bg$-polytope.
We record this formally in the following proposition.

\begin{proposition}\label{prop:dkkgpolytope}
    Given an amply framed DAG, for each maximal clique $M$ we define the \defn{$\bg$-simplex} $\bg(M)$ whose vertices are the $\bg$-vectors of routes of $M$. 
    The set of $\bg$-simplices forms a unimodular triangulation of the $\bg$-polytope of $G$, which we call the \defn{DKK triangulation}.
\end{proposition}

A description of the families of routes that lie on a common facet of the $\bg$-polytope is the following, which for the general case of Gorenstein flow polytopes was proved by Braun and Cornejo~\cite{BraunCornejo}; for the special case of amply framed DAGs, this was independently proved by Berggren~\cite{Berggren}.

\begin{proposition}\label{prop:dagfacesgpolytope}
A set of routes $\mathcal{S}$ lie on a common facet of the $\bg$-polytope of an amply framed DAG $(G,F)$ if and only if for every exceptional route $R$ in $(G,F)$, there is an edge $e_R$ in $R$ such that $e_R$ is not in any of the routes in $\mathcal{S}$.
In other words, there exists a transversal for the exceptional routes that is avoided by $\mathcal{S}$.
\end{proposition}

\begin{figure}
    \centering
    \begin{tikzpicture}
    \begin{scope}[xshift=0, yshift=0, scale=1.25]    
        \vertex[fill](a1) at (1,0) {};
        \vertex[fill](a2) at (2,0) {};
        \vertex[fill](a3) at (3,0) {};
        \vertex[fill](a4) at (4,0) {};

        \draw[-stealth, thick] (a1) to[out=60,in=120]  (a2) node[xshift=-5, yshift=12] {\tiny $2$};
        \draw[-stealth, thick] (a1) to[out=0,in=180] (a2) node[xshift=-10, yshift=4] {\tiny $1$};
        \draw[-stealth, thick] (a1) to[out=60,in=120]  (a3) node[xshift=-3, yshift=11] {\tiny $1$};
        
        \draw[-stealth, thick] (a2) to[out=0,in=180] node[yshift=4] {\tiny $2$} (a3) ;
        \draw[-stealth, thick] (a2) node[xshift=8, yshift=-4.5] {\tiny $1$} to[out=-60,in=-120] (a4);

        \draw[-stealth, thick] (a3) node[xshift=6, yshift=4] {\tiny $1$} to[out=0,in=180] (a4);
        \draw[-stealth, thick] (a3) node[xshift=10, yshift=-4] {\tiny $2$} to[out=-60,in=-120]  (a4);
    \end{scope}

    \begin{scope}[xshift=130, yshift=0, scale=1.25]    
        \vertex[fill](a1) at (1,0) {};
        \vertex[fill](a2) at (2,0) {};
        \vertex[fill](a3) at (3,0) {};
        \vertex[fill](a4) at (4,0) {};

        \draw[-stealth, thick] (a1) to[out=60,in=120]  (a2) node[xshift=-5, yshift=12] {\tiny $2$};
        \draw[-stealth, thick,dashed] (a1) to[out=0,in=180] node[yshift=-6] {} (a2) node[xshift=-10, yshift=4] {\tiny $1$};
        \draw[-stealth, thick,dashed] (a1) to[out=60,in=120] node[yshift=6] {}  (a3) node[xshift=-3, yshift=11] {\tiny $1$};
        
        \draw[-stealth, thick] (a2) to[out=0,in=180] node[yshift=4] {\tiny $2$} (a3) ;
        \draw[-stealth, thick] (a2) node[xshift=8, yshift=-4.5] {\tiny $1$} to[out=-60,in=-120] (a4);

        \draw[-stealth, thick] (a3) node[xshift=6, yshift=4] {\tiny $1$} to[out=0,in=180] (a4);
        \draw[-stealth, thick,dashed] (a3) node[xshift=10, yshift=-4] {\tiny $2$} node[yshift=-10] {} to[out=-60,in=-120]  (a4);
    \end{scope}

    \begin{scope}[xshift=260, yshift=0, scale=1.25]    
        \vertex[fill](a1) at (1,0) {};
        \vertex[fill](a2) at (2,0) {};
        \vertex[fill](a3) at (3,0) {};
        \vertex[fill](a4) at (4,0) {};

        \draw[-stealth, thick] (a1) to[out=60,in=120]  (a2) node[xshift=-5, yshift=12] {};
        \draw[-stealth, thick,dashed] (a1) to[out=0,in=180] node[yshift=-6] {$e_{R_1}$} (a2) node[xshift=-10, yshift=4] {};
        \draw[-stealth, thick,dashed] (a1) to[out=60,in=120] node[yshift=6] {$e_{R_3}$}  (a3) node[xshift=-3, yshift=11] {};
        
        \draw[-stealth, thick] (a2) to[out=0,in=180] node[yshift=4] {} (a3) ;
        \draw[-stealth, thick] (a2) node[xshift=8, yshift=-4.5] {} to[out=-60,in=-120] (a4);

        \draw[-stealth, thick] (a3) node[xshift=6, yshift=4] {} to[out=0,in=180] (a4);
        \draw[-stealth, thick,dashed] (a3) node[xshift=10, yshift=-4] {} node[yshift=-10] {$e_{R_2}$} to[out=-60,in=-120]  (a4);
    \end{scope}

\end{tikzpicture}
    \caption{The framed DAG discussed in Example~\ref{ex:transversal}, including  a collection of dashed edges $\{e_{R_1},e_{R_2},e_{R_3}\}$ that give a transversal of the exceptional routes.}
    \label{fig:facettransversal}
\end{figure}
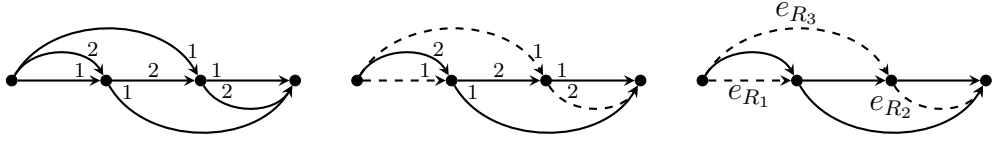

\begin{example}\label{ex:transversal}
The left side of Figure~\ref{fig:facettransversal} depicts an amply framed DAG.
The middle picture is an example of a transversal of the exceptional routes of that DAG. 
The right side of the figure has dashed edges where the two routes formed by the solid edges lie on a common facet of the $\bg$-polytope since they avoid the edges given by $\{e_{R_1},e_{R_2},e_{R_3}\}$.
\end{example}

In the case where an amply framed DAG has disconnected inner graph, the $\bg$-polytope of $G$ decomposes as a free sum of other $\bg$-polytopes as follows.

\begin{theorem}\label{thm:dagdisjointgfreesum}
       If the inner graph of an amply framed DAG $(G,F)$ is disconnected and the edge set of $G$ is the union of the edge sets of the DAGs $G_1, G_2, \ldots, G_t$ where each $G_i$ has inner graph corresponding to a connected component of the inner graph of $G$ and framing $F_i$ induced by $F$, then 
    \[
\bg(G,F) = \bg(G_1,F_1)\oplus \bg(G_2,F_2)\oplus \cdots \oplus \bg(G_t,F_t) \, .
    \]
\end{theorem}

\begin{proof}
    This is an immediate consequence of applying to Definition~\ref{def:gvector} both Proposition~\ref{prop:joinfreesumgorenstein} and Theorem~\ref{thm:dagdisconnectjoin}.
    Note that in this case, since the flow polytope of $G$ is contained in an affine subspace at lattice distance one from the origin, projecting along the linear subspace parallel to the affine span of the special simplex in Theorem~\ref{thm:brunsromer} is equivalent to projecting along the linear span of the special simplex.
\end{proof}

\section{Characterizing amply framed DAGs with locally anti-blocking \texorpdfstring{$\bg$}{g}-polytopes}\label{sec:labdags}

Throughout this section, we let $G$ be a full DAG with no idle edges, let $F$ be an ample framing of $G$, and let $P$ be its $\bg$-polytope. 
We also assume without loss of generality that the inner graph of $G$ is connected, which is sufficient due to Theorems~\ref{thm:dagdisconnectjoin} and~\ref{thm:dagdisjointgfreesum}.
We say that $F$ is a \defn{locally anti-blocking framing} of $G$ if $P$ is locally anti-blocking.
In this section, we will characterize the full DAGs that admit a locally anti-blocking framing.
We begin by setting up needed notation, then proceed to prove our characterization theorem.

\subsection{Path and Cycle DAGs}\label{sec:pathcycledags}

Recall from Section~\ref{sec:flowpolytopes} that an ample framing of a DAG induces a labeling on the edges of $G$ by $\{1,2\}$ such that the pair of incoming edges (and the pair of outgoing edges) have distinct labels.
We will frequently need to consider ample framings where the framing labels are constant on each inner route, leading to the following definition.

\begin{definition}
Let $G$ be an amply framed DAG. 
A \defn{leg} of $G$ is an inner route whose framing labels are constant.
A leg whose framing labels are $1$ is called a \defn{$1$-leg}, while a leg whose framing labels are $2$ is called a \defn{$2$-leg}.
An exceptional route in $G$ is called \defn{long} if it contains a leg of $G$, and \defn{short} otherwise.
\end{definition}

See Example~\ref{eg:legs} for examples.
As a consequence of Theorem~\ref{thm:lab-dag-characterization}  and Lemma~\ref{lem:lab-iff-<3switch} below, in a DAG with a locally anti-blocking framing, the intersection of a route with the inner graph is contained in exactly one leg. 
In this way each route has a single leg on which it resides, leading to the following definition.

\begin{definition}
Given a locally anti-blocking amply framed DAG, we will refer to a route as being \defn{on a $1$-leg} or \defn{on a $2$-leg} if the route resides on a $1$-leg or $2$-leg, respectively.
\end{definition}

We next define path and cycle DAGs, which will play a key role in our classification.

\begin{definition}\label{def:dag-labeling}
Given a sequence $\bk = (k_1, \dots, k_r)$ of positive integers, let \defn{$\Path(\bk)$} denote the full DAG whose inner graph consists of $r$ inner routes of lengths $k_1,\ldots, k_r$ concatenated so that when the inner graph is viewed as an undirected graph, it is an undirected path of length $k_1+\cdots +k_r$.
We note that the special case $\Path(n-1)$ is the \defn{caracol graph} on the vertex set $\{0,\ldots, n+1\}$ with edge multiset $\{(i,i+1) : i=0,\ldots, n \} \cup \{(0,i), (i,n+1)  : i=1,\ldots, n \}$.

Similarly, let \defn{$\Cycle(\bk)$} denote the full DAG whose inner graph consists of $r$ inner routes of lengths $k_1,\ldots, k_r$ concatenated so that when the inner graph is viewed as an undirected graph, it is an undirected cycle of length $k_1+\cdots + k_r$.
Note that $r$ is necessarily even in the definition of $\Cycle(\bk)$.
Also, if $\bk$ is cyclically equivalent to $\bk'$ by a cyclic shift of even length, then $\Cycle(\bk) = \Cycle(\bk')$.
\end{definition}

We will require later in this work a naming convention for edges that are adjacent to the source or sink in the case where every inner route is a leg.
For a DAG where every inner route is a leg, if the source/sink edge is on a short exceptional route and is incident to vertex $i$, we name this edge $i$.
If the source/sink edge is on a long exceptional route and is incident to vertex $i$, then its label is as follows:
\begin{itemize}
    \item $i^-$, for a source edge with framing $2$;
    \item $i^+$, for a sink edge with framing $2$;
    \item $i^+$, for a source edge with framing $1$;
    \item $i^-$, for a sink edge with framing $1$.
\end{itemize}
With this convention, two edges adjacent to the source (or two edges adjacent to the sink) will never have the same label.
However, the two edges on a short exceptional will have the same label.
See Figure~\ref{fig:legs} for an example of our edge-naming conventions.

\begin{example}\label{eg:legs}
Let $\Path(3,4,2)$ have the ample framing given in Figure~\ref{fig:framing_edge_labels}.
Observe that every inner route for this framing is a leg.
We depict $\Path(3,4,2)$ using the same ample framing in Figure~\ref{fig:legs}, where we omit the framing labels and instead include the source- and sink-edge names shown in red using our convention.
Further, an example of a route on a $1$-leg is depicted as a closely dotted line.
An example of a long exceptional route with framing label $2$ is depicted as a solid black line.
An example of a short exceptional route with framing label $1$ is depicted as a loosely dashed line.
\end{example}

\begin{figure}
    \centering
    \includegraphics{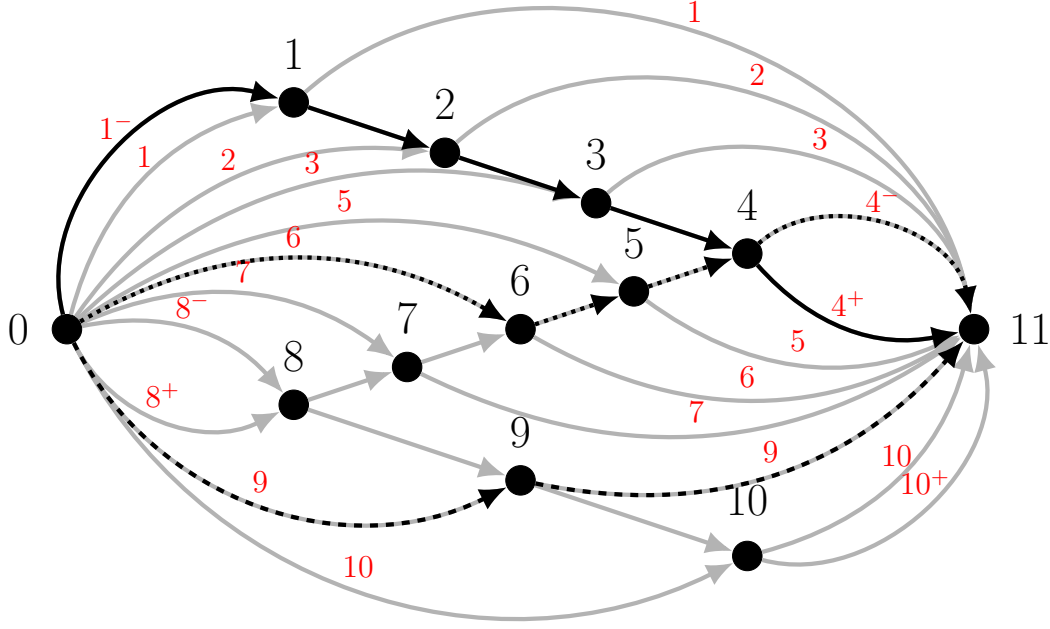}
    \caption{The DAG from Example~\ref{eg:legs}.
    The source-adjacent and sink-adjacent edge labeling of $G$ is shown in red.}
    \label{fig:legs}
\end{figure}

The purpose of distinguishing the long exceptional routes of $G$ is that each such route $R$ functions much like the spine of the graph one gets by restricting $G$ to the subgraph induced by $R$ and all edges incident to $R$. 
The terminology \emph{long} and \emph{short} comes from an equivalent way to define these routes in the locally anti-blocking case which follows from results in the next subsection; in particular, an exceptional route is short if it has exactly two edges, and it is long otherwise.

We next describe two conventions for drawing and labeling $\Path(\bk)$ and $\Cycle(\bk)$.
We call the first convention the \defn{horizontal layout}, and examples are shown in Figure~\ref{fig:lab-dags-path-horizontal}.
For the horizontal layout, we draw the inner path using vertex labels $1,2,\ldots$ where the inner routes are drawn from top to bottom, going left-to-right for the first inner route and switching horizontal direction at each new inner route.
In the resulting drawing, all edges are directed from left to right.
We then add a source visualized on the left and a sink visualized on the right, with source- and sink-edges added as needed to make the resulting DAG full.
For our representations of path DAGs, we will omit the source and sink and only draw half-edges to depict the edges entering and leaving the inner vertices.

We will also introduce the following convention for framing a path DAG in the horizontal layout, which we call the \defn{lab framing}.
Frame the edges in the first, third, fifth, etc.~inner routes, i.e., those with negative slope in our picture, with a $2$.
Frame the edges in the second, fourth, sixth, etc.~inner routes, i.e., those with positive slope in our picture, with a $1$.
This labeling can be completed to an ample framing by labeling the source and sink edges so that when going around any vertex, the labels alternate as shown in Figure~\ref{fig:lab-dags-path-horizontal}.

For the horizontal layout of a cycle DAG $\Cycle(k_1,\ldots, k_r)$, we use an almost identical process, though the final inner route of length $k_r$ must be drawn connecting the lower-right vertex with the upper-left vertex in the figure.
We use an identical framing method for cycle DAGs as for path DAGs.
One may visualize $\Cycle(\bk)$ as if it were embedded on a cylinder so that the slope-interpretation of the labelings of inner edges are consistent with the framing.
See Figure~\ref{fig:lab-dags-path-horizontal} for an example.

\begin{figure}
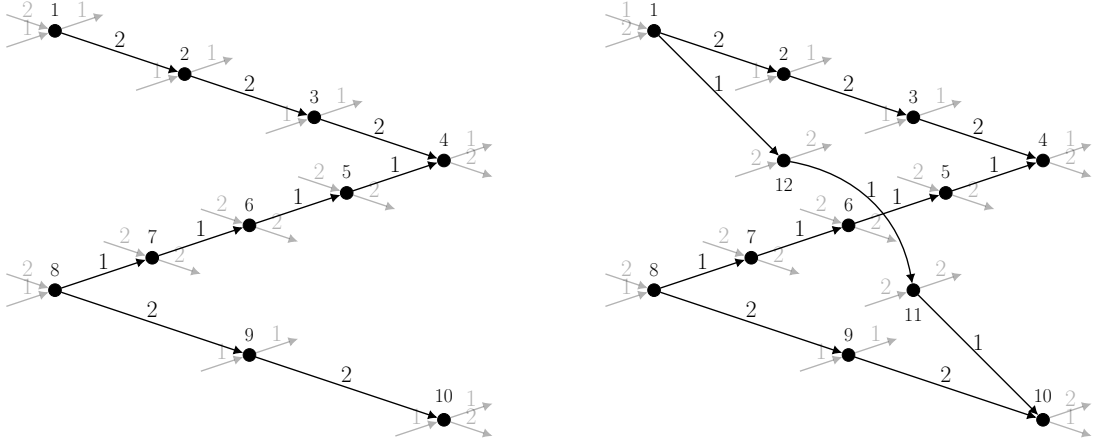

    \centering
    \includegraphics[height=2.3in]{figures/lab_path_framing_horizontal.tex}
    \hspace{1.2cm}
    \includegraphics[height=2.3in]{figures/lab_cycle_horizontal.tex}
     \caption{
    The figure on the left shows the inner graph of $\Path(3,4,2)$ with a horizontal layout, and the figure on the right shows the inner graph of $\Cycle(3,4,2,3)$ with a horizontal layout.
    Both figures show our standard convention for the framing edge labels, which (as we will see) are a locally anti-blocking framing.
    }
    \label{fig:lab-dags-path-horizontal}
\end{figure}

We call the second convention for representing path and cycle DAGs the \defn{circular layout}; examples are shown in Figure~\ref{fig:lab-dags-path-circular}.
For the circular layout, we draw the inner path using vertex labels $1,2,\ldots$ along a circle placed in a counter-clockwise direction (we circle the vertex names to distinguish them from framing and edge labels).
We direct the edges along the first inner route of length $k_1$ from $i$ to $i+1$, and we reverse the direction of the edges each time we switch from one inner route to another in the path or cycle DAG.
We visualize the source as being the center point of the circle, and the sink as being a point at infinity, with source- and sink-edges added as needed to make the resulting DAG full.
As with the horizontal layout of path and cycle DAGs, we will omit the source and sink and only draw half-edges to depict the edges entering and leaving the inner vertices.

We use the same lab framing as previously defined for a path or cycle DAG in the circular layout.
To obtain the lab framing in the circular layout, frame the edges in the first, third, fifth, etc.~inner routes, where we move counter-clockwise starting from vertex $1$, with the label $2$.
Frame the edges in the second, fourth, sixth, etc.~inner routes with a $1$.
This labeling can be completed to an ample framing by labeling the source and sink edges so that when going around any vertex the labels alternate, as shown in Figure~\ref{fig:lab-dags-path-circular}.

\begin{figure}
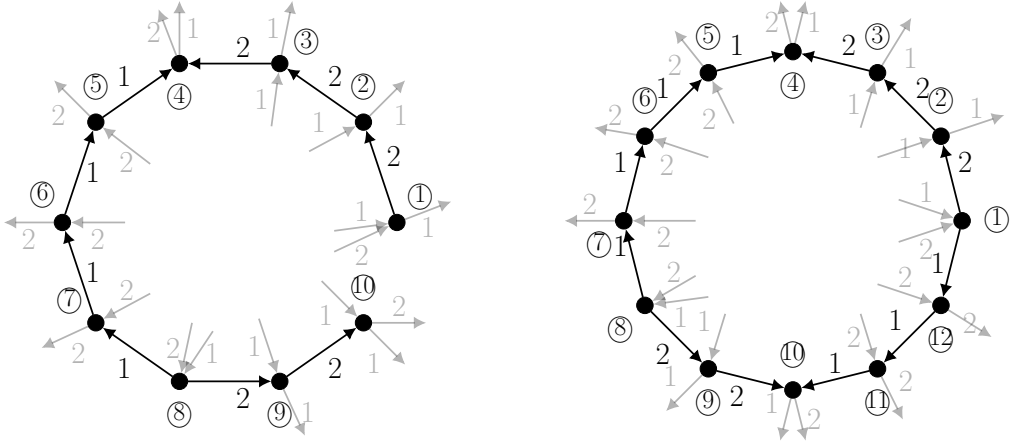

    \centering
    \includegraphics[height=2.3in]{figures/lab_path_framing_circular.tex}
    \hspace{1.2cm}
    \includegraphics[height=2.3in]{figures/lab_cycle_circular.tex}
     \caption{
    The figure on the left shows the inner graph of $\Path(3,4,2)$ with a circular layout, and the figure on the right shows the inner graph of $\Cycle(3,4,2,3)$ with a circular layout.
    Vertex names are circled to distinguish them from the framing labels.
    Both figures show our standard convention for the framing edge labels, which (as we will see) are a locally anti-blocking framing.
    }
    \label{fig:lab-dags-path-circular}
\end{figure}

\subsection{Characterizing DAGs that admit a locally-anti-blocking framing}

The following theorem gives the main result of this section, that the DAGs $\Path(\bk)$ and $\Cycle(\bk)$ are precisely those that produce locally anti-blocking $\bg$-polytopes.

\begin{theorem}\label{thm:lab-dag-characterization} 

If $G$ is a full DAG with no idle edges and connected inner graph, then $G$ admits a locally anti-blocking framing if and only if $G$ is equal to $\Path(\bk)$ or $\Cycle(\bk)$.
\end{theorem}

\begin{proof}
By Lemma~\ref{lem:<3switch-iff-indeg<3} below, $G$ admits a locally anti-blocking framing $F$ if and only if each inner route in $G$ has constant edge labels induced by $F$.
By Lemma~\ref{lem:lab-iff-<3switch} below, each inner route in $G$ has constant edge labels induced by $F$ if and only if its inner graph is a path or a cycle. 
\end{proof}

We next proceed to prove the lemmas invoked in the proof of Theorem~\ref{thm:lab-dag-characterization}.

\begin{lemma}\label{lem:daggpolytopeintegerpoints}
    Let $(G,F)$ be an amply framed DAG and let $P$ be the $\bg$-polytope of $(G,F)$. 
    The only integer points in $P$ are the origin and the vertices of $P$.
\end{lemma}

\begin{proof}
    Observe that $P$ is the image of the flow polytope for $(G,F)$ under the map projecting out the linear span of the exceptional routes.
    Bruns and R\"omer~\cite{BrunsRomer} proved that the integer-point structure of the boundary of $P$ is isomorphic to that of the equatorial complex of the flow polytope (see~\cite{BrunsRomer} for details).
    As every lattice point in the equatorial complex of the flow polytope is a route, it follows that every lattice point in the boundary of $P$ is the image of a route, i.e., a $\bg$-vector.
    Since $P$ is reflexive, the only lattice points in $P$ are the origin and boundary points.
    Since the flow polytope for $G$ is a $0/1$-polytope, every route is a vertex of the flow polytope, hence every integer point in the equatorial complex is a vertex, and hence every non-origin $\bg$-vector in $P$ is a vertex.
\end{proof}

\begin{lemma}\label{lem:lab-iff-<3switch} 
Let $(G,F)$ be an amply framed full DAG without idle edges having $\bg$-polytope $P$.
The following are equivalent.
    \begin{enumerate}[label=(\roman*)]
        \item $P$ is locally anti-blocking.
        \item Each route in $(G,F)$ has edge labels that switch at most twice.
        \item Each inner route in $G$ is a leg.
    \end{enumerate}
\end{lemma}
\begin{proof}
We first show that (i) implies (ii) by a contrapositive argument.
Suppose $G$ has a route $R$ that switches edge labels three or more times. 
By Definition~\ref{def:gvector} it follows that the nonzero coordinates of the $\bg$-vector of a route alternate in sign, hence $\bg(R)$ has at least three nonzero alternating coordinates. 
Towards showing $P$ is not locally anti-blocking, let $i$ be some position of $\bg(R)$ that is not the first nor the last nonzero coordinate and consider the projected vector $\pi_i(\bg(R))$, defined by Proposition \ref{prop:lab-iff-projections}. 
By Lemma \ref{lem:daggpolytopeintegerpoints}, the only integer points of $P$ are the origin or vertices. 
Since the projected vector $\pi_i(\bg(R))$ is not the origin, as it has a nonzero entry, and does not alternate, it cannot be the $\bg$-vector of a route, hence this vector is not in $P$.
Thus, Proposition \ref{prop:lab-iff-projections} implies that the polytope $P$ is not locally anti-blocking.

We next show that (ii) implies (i).
Suppose each route in $(G,F)$ switches edge labels at most twice.
Then each vertex in $P$ has at most two nonzero coordinates.
Let $v$ be such a vertex, with nonzero coordinates indexed by $i$ and $j$ with $i<j$.
We can find a route whose $\bg$-vector is $\pi_i(v)$ by simply following the edge label incoming to vertex $j$ back from $j$ to the source.
Similarly, if we follow the label of the outgoing edge from $i$ until we reach the sink, we have a route whose $\bg$-vector is $\pi_j(v)$.
Thus, along with the origin, we have $\pi_I(v)\in P$ for each $I\in 2^{\{i,j\}}$.
This extends to each $I\in 2^{[n]}$ since the other coordinates are already 0.
Therefore, $P$ is locally anti-blocking by Proposition \ref{prop:lab-iff-projections}.

Lastly, we show that (ii) and (iii) are equivalent.
Suppose $G$ has an inner route that switches labels at least once. 
Since the initial vertex of this route must have two edges incident to the source of $G$, and the final vertex has two edges incident to the sink of $G$, we are free to choose the edge in each case that makes the route switch labels.
This results in a route that switches labels at least three times.
In fact, if every inner route is a leg, the only places a route can switch labels are at the source-adjacent or sink-adjacent vertex, resulting in a maximum of two label switches.
\end{proof}

\begin{example}
Figure~\ref{fig:framing_edge_labels} shows a DAG with edge labels induced by a locally anti-blocking framing.
\end{example}

\begin{lemma}\label{lem:<3switch-iff-indeg<3}
Suppose $G$ is a full DAG with no idle edges and having a connected inner graph.
There is an ample framing of $G$ in which each inner route is a leg if and only if $G$ is equal to $\Path(\bk)$ or $\Cycle(\bk)$.
\end{lemma}

\begin{proof}
First note that since the inner graph of $G$ is connected then it is a path or cycle if and only if each inner vertex is adjacent to at most two other inner vertices of $G$.

Suppose $G$ has an ample framing in which each of its inner routes is a leg, and suppose for contradiction there is an inner vertex $v$ that is adjacent to at least three other inner vertices.
Considering three of the inner edges at $v$, suppose that two are incoming and one is outgoing (which must be the case since $G$ is full) as the other case is symmetric.
Hence, we are free to choose the incoming edge that has the opposite label to the outgoing edge.
Since the three vertices in this two-edge sequence are all inner, this can be extended to an inner route that switches labels, which is not possible.

Now suppose the inner vertices of $G$ induce an undirected path or cycle.
Frame the edges of the inner graph so that each inner route is a leg and so that if two inner routes intersect at a vertex, then those routes receive different labels.
This guarantees that there are distinct labels for two edges incoming to a common vertex, and for two edges outgoing from a common vertex.
Each source and sink edge $e$ then is of one of two types.
If $e$ is the only edge connecting its endpoints, then the framing label on $e$ is uniquely determined in order for the framing to be ample.
If $e$ is one of a pair of parallel edges connecting its endpoints, then the pair of edges can be assigned labels $1$ and $2$, resuling in an ample framing.
\end{proof}

Our previous results yield the following characterization of the ample framings that yield a locally anti-blocking $\bg$-polytope.

\begin{corollary}\label{cor.2framings}
If $G$ is $\Path(\bk)$ or $\Cycle(\bk)$, then an ample framing of $G$ is locally anti-blocking if and only if each inner route of $G$ is a leg.
\end{corollary}

\section{DKK triangulations of \texorpdfstring{$\bg$}{g}-polytopes for locally anti-blocking amply framed DAGs are pulling triangulations}\label{sec:pulling}

For the remainder of this article, we may now assume without loss of generality that all locally anti-blocking DAGs are of the form $\Path(\bk)$ or $\Cycle(\bk)$, each having a lab framing where the first inner route of length $k_1$ has edges labeled by $2$, as in Figure~\ref{fig:lab-dags-path-horizontal}.

Let $P$ be the $\bg$-polytope of a locally anti-blocking amply framed full DAG $(G,F)$ whose inner graph has at least one edge.
By Theorem~\ref{thm:labcompressed} and the reflexivity of $P$, the intersection of every orthant with $P$ is compressed.
Since the DKK triangulation of $P$ (as described in Proposition~\ref{prop:dkkgpolytope}) is unimodular and, as we will see, refines the subdivision of $P$ induced by pulling $P$ at the origin, it is natural to ask if the DKK triangulation of $P$ can be obtained as a pulling triangulation.
We give an affirmative answer to this question in Theorem~\ref{thm:dkkpulling}.
More generally, we provide an algorithm for producing an irredundant pulling order, and using this algorithm determine all possible pulling orders for constructing the DKK triangulation of $P$.

\subsection{Minimal faces}\label{sec:minimalfaces}

In light of the considerations in Section~\ref{sec:pulling-triangulations}, if we want to realize the DKK triangulation of the $\bg$-polytope via a pulling order on its vertices, then we need to ensure that whenever we pull a route, by which we mean pull the $\bg$-vector of the route, it is coherent with every route that it is cell-neighboring in the subdivision.
This means that the first route that we pull has to be coherent with all other routes.
In other words, we must start by pulling $\bzero$, which is the $\bg$-vector of the exceptional routes of $(G,F)$.
This subdivides the $\bg$-polytope into the cells $\conv(F,\bzero)$ for each facet $F$ of $P$.

Following this initial pull at the origin, Theorem \ref{thm:pulling1}  characterizes when pulling a point separates two vertices and emphasizes that we must investigate what the minimal face containing any two routes in conflict looks like.
The following characterization of the faces of the $\bg$-polytope, which is a restatement of a special case of results due to Braun and Cornejo~\cite[Theorem 4.6]{BraunCornejo} and Berggren~\cite{Berggren}, allows us to do just that.

\begin{definition}
    Let $Y$ be a subset of the routes of given DAG. We define 
    \[
    E_Y :=\cup_{R\in Y}\{e\in R\} \, .
    \]
\end{definition}

\begin{proposition}[Braun and Cornejo~\cite{BraunCornejo}, Berggren~\cite{Berggren}]\label{prop:faces-edge-sets}
    The faces of the $\bg$-polytope are in bijection with edge sets of the form $E_Y$ that do not contain all edges of any exceptional route, where $Y$ is a subset of the routes of $G$.
    Specifically, given such an $E_Y$, the convex hull of the indicator vectors of all routes contained in $E_Y$ forms a face $Q_Y$ of the flow polytope. 
    Further, the map $\phi$ is an integral equivalence between $Q_Y$ and the face $\phi(Q_Y)$ of the $\bg$-polytope corresponding to $E_Y$.
\end{proposition}

It is an immediate corollary that every non-exceptional route is a face, hence a vertex.

\begin{corollary}
    \label{cor:vertices}
    Let $P$ be the $\bg$-polytope of an amply framed full DAG $(G,F)$.
    The $\bg$-vector for every non-exceptional route in $G$ is a vertex of $P$.
\end{corollary}

Note that in Proposition~\ref{prop:faces-edge-sets}, the bijection is not with the set of routes $Y$ but rather with the edges in the union.
Thus, different sets of routes $Y$ could index the same set of edges $E_Y$ corresponding to a face.
Proposition~\ref{prop:faces-edge-sets} allows us to work with indicator vectors of routes in the flow polytope when considering faces of the $\bg$-polytope, as the following corollary demonstrates.

\begin{corollary}\label{cor:components-and-minimal-face}
    Let $P$ be the $\bg$-polytope of an amply framed full DAG $(G,F)$. Suppose $R$ and $S$ are two routes in $G$ where $E_{\{R,S\}}$ does not contain an exceptional route, and let $\kappa$ be the number of connected components of $R\cap S$, where a component might consist of only a single vertex. 
    Then the minimal face of $P$ containing both $R$ and $S$ is integrally equivalent to a $(\kappa-1)$-cube.
\end{corollary}
\begin{proof}
    The edge union $E_{\{R,S\}}$ is the unique minimal edge set by containment that contains the routes $R$ and $S$. Hence, by Proposition~\ref{prop:faces-edge-sets}, the minimal face of $P$ containing $R$ and $S$ is equivalent to the convex hull of the indicator vectors of all the routes contained in $E_{\{R,S\}}$. 
    Suppose $R\cap S$ has components $C_1,C_2,\ldots,C_\kappa$ (where the source is in $C_1$ and the sink in $C_{\kappa}$) with segments $R_i$ and $S_i$ connecting $C_i$ to $C_{i+1}$ along $R$ and $S$, respectively.
    Fix an edge $h_i\in R_i$ for each $i$.
    Then a route in $E_{\{R,S\}}$ is uniquely determined by assigning a value of $0$ or $1$ to each $h_i$, and completing the indicator vector for a route from this labeling.
    Further, any point in the face defined by $E_{\{R,S\}}$ is uniquely determined by assigning each edge $h_i$ a value $\lambda_i\in [0,1]$.
    Thus, the face defined by $E_{\{R,S\}}$ injectively projects onto the subspace defined by the coordinates indexed by the $h_i$'s, and the projection forms a unit cube in $\R^{\kappa-1}$.
\end{proof}

In the following lemma, we consider faces of the $\bg$-polytope in the locally anti-blocking case.

\begin{lemma}\label{lem:number-components-of-route-intersection}
    Let $G$ be $\Path(\bk)$ or $\Cycle(\bk)$ for some sequence $\bk = (k_1, \dots, k_r)$ of positive integers. Suppose $R$ and $S$ are distinct routes in $G$ and let $\kappa$ be the number of connected components of $R\cap S$. Then:
    \begin{enumerate}[label=(\roman*)]
        \item We have $2\leq \kappa\leq 4$.
        \item We have $\kappa=4$ if and only if $G = \Cycle(k_1, k_2)$ and $E_{\{R,S\}} = E_{\{L_1, L_2\}}$, where $L_1$ and $L_2$ are the two long exceptional routes. 
        \item If $E_{\{R,S\}}$ does not contain an exceptional route, then $\kappa\leq 3$. 
        Further:
        \begin{enumerate}
        \item If $R$ and $S$ are in conflict, then $\kappa=3$.
        \item If $R$ and $S$ are coherent, then $\kappa=2$.
        \end{enumerate}
    \end{enumerate}
\end{lemma}
\begin{proof}
    Since $R$ and $S$ are distinct and intersect at the source and sink, it follows that $\kappa\geq 2$. It is a direct consequence of Lemma~\ref{lem:lab-iff-<3switch}~\emph{(ii)} that $\kappa\leq 4$. 
    However, $\Cycle(k_1,k_2)$ is the only locally anti-blocking framed DAG in which there is a pair of inner vertices with two different paths between them, namely the vertices 1 and $k_1$ having exactly two paths between them. 
    Therefore, the only way $R\cap S$ can have four components is when the components are the single vertices $s$, $1$, $k_1$, and $t$. 
    Between each of these components there are exactly two paths, and if $R\cap S = \{s,1,k_1,t\}$, then $R$ and $S$ take different paths between each vertex. Hence, $E_{\{R,S\}}$ is equal to the union of all six of these paths, equivalently stated in \emph{(ii)}.
    In this case $E_{\{R,S\}}$ contains an exceptional route, thus when $E_{\{R,S\}}$ does not contain an exceptional route, we have $\kappa\leq 3$. 
    If $R$ and $S$ are in conflict, then $\kappa$ must be equal to $3$, since two routes cannot conflict on the source- or sink-containing component.
    If the two routes are coherent, then they overlap on only the two components containing the source and sink, and thus $\kappa=2$.
\end{proof}

The following corollary shows that for locally anti-blocking framed full DAGs, the minimal face of the $\bg$-polytope containing a pair of conflicting routes is either the $\bg$-polytope or a quadrilateral.

\begin{corollary}\label{cor:incoherence-square}
    Let $(G,F)$ be a locally anti-blocking framed full DAG, let $P$ be its $\bg$-polytope, and suppose $R$ and $S$ are routes.
    Let $f$ be the minimal face of the $\bg$-polytope containing their $\bg$-vectors. 
    Then:
    \begin{enumerate}
    [label=(\roman*)]
        \item If $E_{\{R,S\}}$ contains the edge set of an exceptional route, then $f$ is the entire polytope $P$ (the routes lie on different facets).
        \item Otherwise:
        \begin{enumerate}
        \item If $R$ and $S$ are in conflict, then $E_{\{R,S\}}$ contains four distinct routes and $f$ is the quadrilateral face that is the convex hull of their $\bg$-vectors.
        \item If $R$ and $S$ are coherent, then $E_{\{R,S\}}$ contains two distinct routes and $f$ is the edge formed by $R$ and $S$.
        \end{enumerate}
    \end{enumerate}
\end{corollary}

\begin{proof}
    Suppose that $E_{\{R,S\}}$ contains an exceptional route. Then there is no edge set containing the edges of $R$ and $S$ that does not contain the edges of an exceptional route. By Proposition~\ref{prop:faces-edge-sets}, there is no proper face of $P$ containing $R$ and $S$, implying that the minimal face of $P$ containing them is $P$ itself.
    Now we may assume that $E_{\{R,S\}}$ does not contain an exceptional route. The proof is completed by applying Lemma \ref{lem:number-components-of-route-intersection} followed by Corollary \ref{cor:components-and-minimal-face} to arrive at \textit{(ii)}.
\end{proof}

\subsection{Route Pairs}\label{sec:routepairs} 

In order to use our preceding results to identify pulling orders yielding the DKK triangulation, we will need a combinatorial model that encodes when routes in a locally anti-blocking framed DAG are coherent.
We begin by introducing a notational convention.
Since the inner graph of a locally anti-blocking framed DAG $G$, as described in Section~\ref{sec:labdags}, forms a path or cycle, each route of $G$ is completely determined by its source-incident edge, sink-incident edge, and the label of the inner edges (if any) used by the route.

\begin{definition}\label{def:routepairlabels}
Let $G$ be a locally anti-blocking framed DAG.
For a short exceptional route $S$ passing through the inner vertex $i$, the \defn{route pair defining $S$} is the ordered pair $(i,i)$.
Let any other route $R$ have source edge $a$ and sink edge $b$.
If $R$ is on a $2$-leg, the \defn{route pair defining $R$} is the ordered pair $(a,b)$. 
If $R$ is on a $1$-leg, the \defn{route pair defining $R$} is the ordered pair $(b,a)$.
\end{definition}

\begin{example}\label{ex:routepairlabels}
    Consider the DAG shown in Figure~\ref{fig:legs}, using the lab framing.
    The route following the vertex sequence $[0,2,3,11]$ is labeled by the pair $(2,3)$.
    The route following the vertex sequence $[0,3,11]$ is labeled by the pair $(3,3)$.
    The route following the vertex sequence $[0,7,5,11]$ is labeled by the pair $(5,7)$.
    The route following the vertex sequence $[0,8,7,6,5,4,11]$ using the edges $8^-$ and $4^+$ is labeled by the pair $(4^+,8^-)$.
\end{example}

Observe that when a DAG is presented with the circular layout, then the route pair lists the source/sink edges in the order they appear when traversing the inner DAG in a counterclockwise manner, as shown in the next example.

\begin{example}
    \label{ex:routepaircircular}
    Consider the right DAG in Figure~\ref{fig:lab-dags-path-circular}, which is identical to the right DAG in Figure~\ref{fig:lab-dags-path-horizontal}.
    The route from the source to vertices $7$ then $6$ then to the sink is denoted by the pair $(6,7)$.
    This agrees with the order of vertices $6$ and $7$ in the counter-clockwise order in the circular layout.
    On the other hand, the route from the source to vertex $2$ to vertex $4$ to the sink is denoted by the pair $(2,4^+)$.
    This again has the order of the underlying vertices for the edge labels in the pair appear in counter-clockwise order.
\end{example}

For most locally anti-blocking framed DAGs, it is sufficient to name only the unordered pair of source- and sink-edge to uniquely identify a route.
However, the convention above allows us to differentiate between the routes in $\Cycle(k_1,k_2)$ that have the same source and sink edges.
For instance, in $\Cycle(k_1,k_2)$, the pair $(1^+, k_1^-)$ refers to the route lying on a $2$-leg with these source and sink edges while $(k_1^-, 1^+)$ refers to the route lying on a $1$-leg with the same source and sink edges.

\begin{example}
\label{ex:smallcyclelabels}
Consider the framed DAG given by $\Cycle(3,2)$ in Figure~\ref{fig:cycletwolabels}.
We will use the source and sink edge labeling conventions given in Subsection~\ref{sec:pathcycledags}, though only two of these edge labels ($1^+$ and $4^-$) are listed in the figure.
Note that the unordered pair of edge labels $\{2,3\}$ uniquely determines the route that goes from the source to vertex $2$, followed by vertex $3$, followed by the sink.
However, the unordered pair of edge labels $\{1^+,4^-\}$ is ambiguous in this case, as it is not clear whether the route should enter along edge $1^+$ and traverse the upper or lower path to exit at $5^-$.
Thus, our route pair labeling convention allows us to designate $(1^+,4^-)$ as the route that traverses the upper inner route with framing label $2$, while $(4^-,1^+)$ is the route that traverses the lower inner route with framing label $1$.
\end{example}

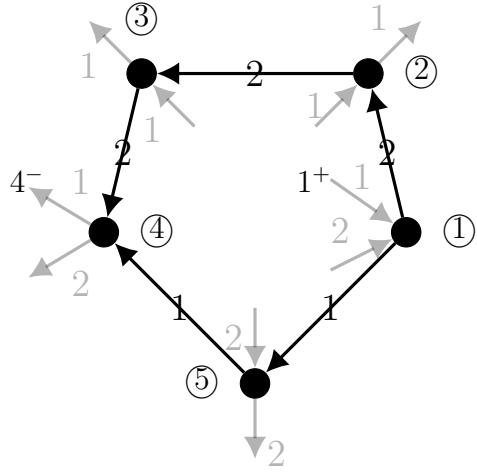
\begin{figure}
    \centering
\begin{tikzpicture}
\begin{scope}[every node/.style={circle,fill,draw,above}]
    \node[label=right:\textcircled{$1$}] (s) at (2,0) {};
    \node[label=right:\textcircled{$2$}] (u1) at (3/2,2.1) {};
    \node[label=above:\textcircled{$3$}] (u2) at (-3/2, 2.1) {};
    \node[label=right:\textcircled{$4$}] (t) at (-2,0) {};    
    \node[label=left:\textcircled{$5$}] (l2) at (0,-2) {};
\end{scope}
    
\begin{scope}
    \node[label=left:$1^+$] (ss) at (1.3,.9) {};
    \node[label=left:$4^-$] (tt) at (-2.5,.9) {};
\end{scope}

\begin{scope}[every edge/.style={draw,very thick}]
    \foreach \u/\v in {s/u1, u1/u2, u2/t}
        \path [-{Latex[width=3mm]}] (\u) edge node {\large 2} (\v);
    \foreach \u/\v in {s/l2, l2/t}
        \path [-{Latex[width=3mm]}] (\u) edge node {\large 1} (\v);
\end{scope}

\begin{scope}[every edge/.style={draw,very thick},
              opacity=.3]
    \path [{Latex[width=3mm]}-] (s) edge node[style=above] {\large 1} ++(-1,.7);
    \path [{Latex[width=3mm]}-] (s) edge node[style=above left] {\large 2} ++(-1,-.5);
    \path [-{Latex[width=3mm]}] (u1) edge node[style=above left] {\large 1} ++(.7,.7);
    \path [{Latex[width=3mm]}-] (u1) edge node[style=left] {\large 1} ++(-.7,-.7);
    \path [{Latex[width=3mm]}-] (u2) edge node[style=below left] {\large 1} ++(.7,-.7);
    \path [-{Latex[width=3mm]}]  (u2) edge node[style=below left] {\large 1} ++(-.7,.7);
    \path [-{Latex[width=3mm]}]  (t) edge node[style=below right] {\large 2} ++(-1,-.6);
    \path [-{Latex[width=3mm]}]  (t) edge node[style=above right] {\large 1} ++(-1,.6);
    \path [-{Latex[width=3mm]}] (l2) edge node[style=below right] {\large 2} ++(0,-1);
    \path [{Latex[width=3mm]}-]  (l2) edge node[style=left] {\large 2} ++(0,1); 
\end{scope}
\end{tikzpicture}

    \caption{The cycle $\Cycle(3,2)$ with a lab framing.}
    \label{fig:cycletwolabels}
\end{figure}

Lemma \ref{lem:number-components-of-route-intersection} makes it clear that when the inner graph is a cycle consisting of two legs, as in Figure~\ref{fig:cycletwolabels}, we must make special considerations for the eight routes contained in the edge union of the two long exceptional routes. 
We will occasionally have to make an exception for this case in the statements of our results, hence the following definition.

\begin{definition}
    In a locally anti-blocking DAG with inner DAG $\Cycle(k_1,k_2)$, the eight routes contained in the edge union of the two long exceptional routes are called \defn{anomalous routes}.
\end{definition}

\begin{example}\label{ex:anomalousroutes}
    In Figure~\ref{fig:cycletwolabels}, the eight routes given by the ordered pairs in
    \[
    \{1^+,1^-\}\times \{4^+,4^-\}\bigcup \{4^+,4^-\}\times \{1^+,1^-\}
    \]
    are the anomalous routes.
\end{example}

The order of the names given to the source/sink edges plays a role in the combinatorial descriptions to come. 
Thus we must impose an order on the superscripted integers amid the usual total order on the unmarked integer labels, which we do as follows, where the inner vertices have labels ranging from $1$ to $n$:    
\begin{equation}\label{eqn:order}
1^- < 1 < 1^+ < \cdots< (i-1)^+ < i^- < i < i^+ < (i+1)^- < \cdots < n^+ \quad .
\end{equation}
Note that in a typical locally anti-blocking DAG, only a subset of these edge labels will be used.
What is important is to impose upon them this relative order.

When $G=\Cycle(\bk)$, we consider this to be a cyclic order.
In most cases, we will only be comparing the order of indices that are on the same leg or adjacent legs, and hence we can restrict to the linear order on that sequence of legs. 
However, in the case where the inner graph is a cycle with only two inner legs it becomes necessary to take into account the cyclic nature of the order.
This will be addressed further as it arises.

Note that, by the structure of a locally anti-blocking DAG and our previous labeling conventions of vertices and edges, for almost any route $(a,b)$ we have $a\leq b$ when we restrict the order to the source/sink edges that are incident to the leg on which the route lies.
The exception to this is routes involving the edge labels $1^-$ and $1^+$ that reside on the final leg of a cycle, where we have the labels with a $1$ occurring in the second position.

\subsection{Coherence Diagrams}\label{sec:coherencediagrams}

We can represent the set of routes of $G$ using a diagram that, as we will see, encodes the coherence relation for pairs of routes.

\begin{definition}\label{defn:coherence_diagram}
Let $G$ be a locally anti-blocking DAG with inner DAG of the form $\Path(\bk)$ or $\Cycle(\bk)$.
The \defn{coherence diagram} of $G$ is an array of boxes defined as follows, whose rows are indexed top-to-bottom by the labels of the source-adjacent edges of $G$ using the linear order~\eqref{eqn:order} and whose columns are indexed left-to-right by the labels of the sink-adjacent edges using the linear order~\eqref{eqn:order}.
Let $G$ be a locally anti-blocking framed DAG with inner DAG given by $\Path(\bk)$.
For each route $R$ in $G$, if $R$ is on a $2$-leg, then we include the box in row $a$ and column $b$ where $(a,b)$ is the route pair for $R$.
If $R$ is on a $1$-leg, then we include the box in row $b$ and column $a$ where $(a,b)$ is the route pair for $R$.

When $G$ is a locally anti-blocking framed DAG with inner DAG given by $\Cycle(\bk)$, the coherence diagram is constructed similarly, however we include two additional rows at the bottom of the diagram, where these rows are indexed by the labels of the two source-incident edges at vertex $1$ (repeating the order of the labeling of the top two rows).
For a route supported on a $1$-leg with source edge labeled by $1^+$ or $1^-$, we place the box corresponding to this route in the bottom two rows.
We then identify the left-most boundary edges in the top two rows with the right-most boundary edges in the bottom two rows, concatenating the top two rows with the bottom two rows.
\end{definition}

The following three examples of coherence diagrams provide clarification regarding the correspondence between row/column indices for boxes and the entries of the route pairs.
One way to visualize the coherence diagram, as seen in Figures~\ref{fig:path342diagram},~\ref{fig:cycle3423diagram}, and~\ref{fig:cycle32diagram} and formalized in Lemma~\ref{lem:outercorners}, is that each row/column of the diagram contains precisely one box corresponding to an exceptional route.
This follows from the facts that we have an ample framing and that every inner vertex has in- and out-degree $2$, and hence every edge lies on a unique exceptional route~\cite{Polytopes2}.
The set of exceptional routes with coordinates $(i,i)$ form lower-left and upper-right staircases in the coherence diagram.
If a box is to the right or above a lower-left staircase, then it corresponds to a route pair on a route lying on a $2$-leg.
If a box is to the left or below an upper-right staircase, then it corresponds to a route pair on a route lying on a $1$-leg.
The long exceptional routes form isolated corners of the diagram on the ``opposite'' side from each staircase.

\begin{example}
Figure~\ref{fig:path342diagram} is the coherence diagram for $\Path(3,4,2)$.  
Observe that there is a row and a column indexed by $6$, as $6$ is an inner vertex connected to both the source and the sink.
However, there is not a column labeled $8^+$, since vertex $8$ is not connected to the sink.
Further, the box in row $1$ and column $3$ corresponds to the unique route that passes through the source edge labeled $1$ and the sink edge labeled $3$, which is given by the route pair $(1,3)$.
The box in row $7$ and column $5$ corresponds to the unique route that passes through the source edge labeled $7$ and the sink edge labeled $5$, which is given by the route pair $(5,7)$.
\end{example}

\begin{example}
Figure~\ref{fig:cycle3423diagram} is the coherence diagram for $\Cycle(3,4,2,3)$.
We have indicated the identification of the upper-left pair of vertical edges with the lower-right pair of vertical edges by highlighting them in magenta.
Thus, for example, in the center figure we consider the box in row $1^-$ and column $2$, which is given by the route pair $(1^-,2)$, to be adjacent along an edge to the box in row $1^-$ and column $12$, which is given by the route pair $(12,1^-)$.
\end{example}

\begin{example}
Figure~\ref{fig:cycle32diagram} is the coherence diagram for the DAG given in Figure~\ref{fig:cycletwolabels}, i.e., for $\Cycle(3,2)$.
Observe that in this coherence diagram, there are two boxes corresponding to row $1^-$ and column $4$.
The upper box with these coordinates is given by the route pair $(1^-,4)$, as that route sits on a $2$-leg, while the lower box with these coordinates is given by the route pair $(4,1^-)$, as that route sits on a $1$-leg.
This use of repeated coordinates for boxes in the coherence diagram of $\Cycle(k_1,k_2)$ is the primary motivation for our ordered pair conventions for routes.
In Figure~\ref{fig:cycle32diagram}, we have indicated the boxes corresponding to anomalous routes with diamonds, which are the eight anomalous routes identified in Example~\ref{ex:anomalousroutes}.
\end{example}

\begin{figure}
    \centering
    \includegraphics{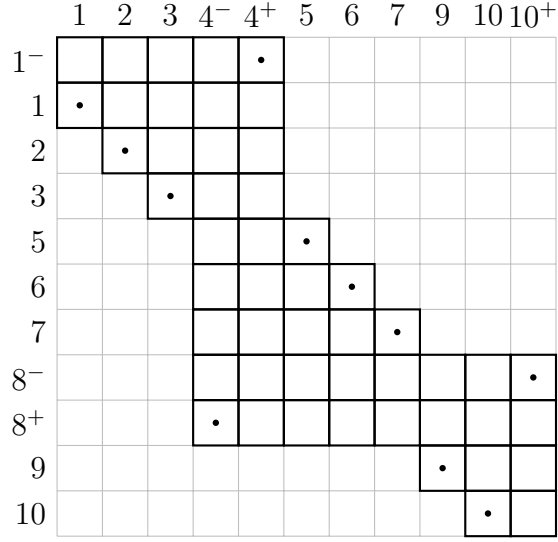}
    \caption{The coherence diagram corresponding to $\Path(3,4,2)$.
    Boxes with a dot inside indicate exceptional routes.
    }
    \label{fig:path342diagram}
\end{figure}

\begin{figure}
    \centering
    \includegraphics{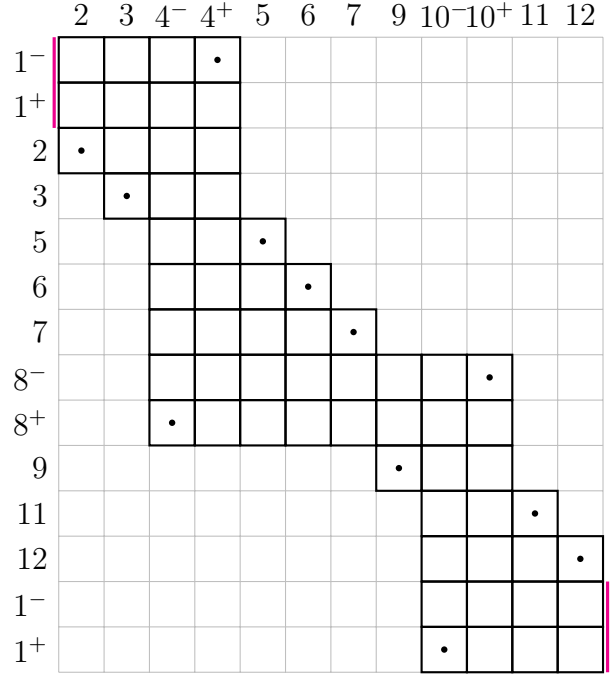}
    \caption{The coherence diagrams corresponding to $\Cycle(3,4,2,3)$.
    Boxes with a dot inside indicate exceptional routes.
    }
    \label{fig:cycle3423diagram}
\end{figure}

\begin{figure}
    \centering
    \includegraphics{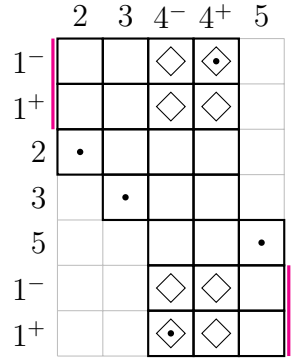}
    \caption{The coherence diagram corresponding to $\Cycle(3,2)$.
    Boxes with a dot inside indicate exceptional routes.
    Boxes containing a diamond indicate anomalous routes.
    }
    \label{fig:cycle32diagram}
\end{figure}

\begin{remark}
The idea of the coherence diagram previously appeared in~\cite{BGMY23} as the collection of lattice points above a lattice path for a family of DAGs called $\nu$-caracol graphs.
The case of the classical caracol $\Path(n)$ appearing in this article corresponds to the case $\nu=(1,\ldots,1)$, and any collection of coherent boxes in the coherence diagram of $\Path(\bk)$ corresponds to a collection of tree-compatible points in the setting of~\cite{BGMY23}.
Thus, a maximal collection of coherent boxes in the coherence diagram of $\Path(\bk)$ is a grid representation of a binary tree, see~\cite{CPS19}.
\end{remark}

\begin{remark}\label{remk:ar-quiver}
    The work~\cite{Polytopes2} associates to an amply framed DAG $(G,F)$ a gentle algebra $\Lambda(G,F)$ whose $\tau$-tilting theory describes the DKK triangulation of $(G,F)$.
    When $G$ is a locally anti-blocking DAG with the lab framing $F$, the coherence diagram for $G$ can be thought of as an extended Auslander-Reiten quiver of the corresponding gentle Nakayama algebra $\Lambda(G)$. Here boxes without a dot, i.e., those corresponding to non-exceptional routes, are in bijection with indecomposable $\Lambda(G,F)$-modules and shifted projective $\Lambda(G,F)$-modules.   The irreducible morphisms, or arrows of the Auslander-Reiten quiver, move directly right and directly down between adjacent boxes of the coherence diagram.   Moreover, applying the Auslander-Reiten translation $\tau$ corresponds to moving up-left along a diagonal in the coherence diagram.
\end{remark}

We next introduce several definitions that will be needed throughout this work.

\begin{definition}\label{def:outercorner}
We say a box in the coherence diagram is an \defn{outer corner} if it is the top right or bottom left box for some maximal (under containment) rectangle in the diagram, including rectangles that cross the identified boundary line in the cycle case.    
\end{definition}

\begin{definition}\label{def:boundingstrip}
Given a locally anti-blocking DAG with inner DAG either $\Path(\bk)$ or $\Cycle(\bk)$, let $I_s$ denote the set of edges adjacent to the source $s$ that are part of a parallel pair of edges between $s$ and an inner vertex.
Define $I_t$ similarly for edges adjacent to the sink $t$ that are part of a parallel pair of edges.
We define the \defn{bounding strip}, denoted by $\lambda_{s,t}$, to be the subdiagram of the coherence diagram for $G$ given by the union of the rows indexed by $I_s$ and the columns indexed by $I_t$.
\end{definition}

If $e, e'\in I_s$ are adjacent to a common vertex, then the rows indexed by $e$ and $e'$ form a $2\times m$ rectangle in this subdiagram, where $m$ is the number of routes in $G$ that contain $e$.
Similarly, such a pair of edges in $I_t$ forms an $m\times 2$ rectangle.
Note that in the case of a cycle, the rows corresponding to $1^-$ and $1^+$ are split between the top and bottom rows of the diagram.

For example, in Figure~\ref{fig:path342diagram}, the columns indexed by $4^-$ and $4^+$ form a $9\times 2$ rectangle and the rows indexed by $8^-$ and $8^+$ form a $2\times 8$ rectangle.
In the center diagram corresponding to $\Cycle(3,4,2,3)$, the rows for $1^-$ and $1^+$ are split into two $2\times 4$ rectangles, one occurring at the top of the diagram and one at the bottom of the diagram; these correspond to routes supported on the two legs of length $3$ that contain vertex $1$.

As the following lemma shows, the coherence diagram is a skew Ferrers diagram (using the French convention).

\begin{lemma}\label{lem:outercorners}
The coherence diagram for $G$ is a skew Ferrers diagram with outer corners in bijection with the exceptional routes of $G$.
\end{lemma}

\begin{proof}
The outer corners of $\lambda_{s,t}$ correspond to the long exceptional routes.
On the inner side of each turn in $\lambda_{s,t}$, the coherence diagram contains a staircase partition diagram, in which the boxes correspond to routes that are supported on a fixed leg but contain no edges from $I_s\cup I_t$.
Unioning $\lambda_{s,t}$ with these staircase diagrams produces the full coherence diagram.
The outer corners of these staircase diagrams correspond to the short exceptional routes.
The condition regarding maximal rectangles defining outer corners follows from this description of the coherence diagram.
\end{proof}

\begin{remark}
    Mészáros, Morales, and Striker \cite{MMS19} give an integral equivalence between flow polytopes and order polytopes in the case of strongly planar DAGs and posets using a planar duality construction.
    Using this construction, for $G = \Path(\bk)$ drawn in the horizontal layout, the flow polytope $\calF_1(G)$ will be integrally equivalent to an order polytope of a generalized snake poset $P$, where generalized snake posets are a family of width-two posets whose order polytopes were studied by von Bell et al.~\cite{gensnakeposets}.
    Moreover, one can see the coherence diagram for $\Path(\bk)$ as the Hasse diagram of a poset, where elements are the boxes and cover relations are given by boxes sharing an edge oriented from left to right and top to bottom.
    In this setting, the lattice of order ideals $J(P)$ for the strongly planar poset dual to $\Path(\bk)$ will be isomorphic to the coherence diagram as a poset. 
\end{remark}

We will now define various regions of the coherence diagram with respect to a route, which we will use to combinatorially characterize coherence in the next section.

\begin{definition}\label{def:diagramregions}
    Let $R$ be a route in a locally anti-blocking framed DAG.
    The sets $\north(R)$, $\south(R)$, $\east(R)$, and $\west(R)$ are the regions of the coherence diagram strictly to the north, south, east, and west of $R$ respectively. I.e., $\north(R)$ consists of the boxes above $R$ in its column, $\east(R)$ consists of the boxes to the right of $R$ in its row, and so on.
    In the case of a cycle, when $R$ is in either the top or bottom two rows (which are identified along the lower right and upper left boundary), the regions $\east(R)$ and $\west(R)$ continue across the line of identification.

    We define four additional regions.      \begin{itemize}
            \item The region $\NW(R)$ consists of those boxes $R'$ such that there exists boxes $T_1$ and $T_2$ such that $\{T_1\}=\west(R) \cap \south(R')$ and $\{T_2\}=\north(R) \cap \east(R')$.
            \item The region $\NE(R)$ consists of those boxes $R'$ such that there exist boxes $T_1$ and $T_2$ such that $\{T_1\}=\east(R) \cap \south(R')$ and $\{T_2\}=\north(R) \cap \west(R')$.
            \item The region $\SW(R)$ consists of those boxes $R'$ such that there exist boxes $T_1$ and $T_2$ such that $\{T_1\}=\west(R) \cap \north(R')$ and $\{T_2\}=\south(R) \cap \east(R')$.
            \item The region $\SE(R)$ consists of those boxes $R'$ such that there exist boxes $T_1$ and $T_2$ such that $\{T_1\}=\east(R) \cap \north(R')$ and $\{T_2\}=\south(R) \cap \west(R')$.
        \end{itemize}
    We say that a route $S$ is \defn{northwest} of $R$ if $S\in \NW(R)$, and similarly for northeast, southwest, and southeast.
    Supposing $R$ is northwest of $S$, the \defn{2-by-2 array} subtended by $R$ and $S$, denoted $R\squarette S$, is the set $\{R, S, T_1, T_2\}$ where $T_1$ and $T_2$ are the routes required to define $R$ as northwest of $S$. Thus the notation $R\squarette  S$ implies $R\in \NW(S)$.
\end{definition}

\begin{example}
    Consider the diagram for $\Path(3,4,2)$ in Figure~\ref{fig:path342diagram}.
    The box for route pair $(4^-,5)$ (in row $5$, column $4^-$) is northwest of the box for route pair $(5,7)$ (in row $7$, column $5$), since we can see $\west(5,7)\cap \south(4^-,5)=\{(4^-,7)\}$ and $\north(5,7)\cap \east(4^-,5)=\{(5,5)\}$.
    However, the box for $(3,4^-)$ (in row $3$, column $4^-$) is not northwest of the box for $(5,7)$ since $\north(5,7)\cap \east(3,4^-)=\emptyset$.
\end{example}

Note that by definition we have $R\in \NW(S)$ if and only if $S\in \SE(R)$.
Also, we have $R\in \NE(S)$ if and only if $S\in \SW(R)$.
By design of the coherence diagram, for any route $R$ we have $\north(R)\cap \south(R) = \emptyset$ and $\west(R) \cap \east(R)=\emptyset$.
However, the pairwise mutual exclusivity of the regions $\NW, \NE, \SW$ and $\SE$ is only guaranteed when $G\neq \Cycle(k_1,k_2)$.
The exceptions to this in the case of $\Cycle(k_1,k_2)$ involve the anomalous routes.
For instance, if $R = (1^+, k_1^-)$ and $R' = (k_1^+, 1^-)$ then $R \in \NW(R')\cap \SE(R')$.

\begin{example}\label{ex:anomalousbadbehavior}
In diagram for $\Cycle(3,2)$ in Figure~\ref{fig:cycle32diagram}, the box with route pair $(1^+,4^-)$ is both northwest and southeast of the box with route pair $(4^+,1^-)$.
Even worse, the box with route pair $(1^-,4^-)$ is both northwest and northeast of the box with route pair $(4^+,1^+)$.
This violation of the exclusivity of regions only occurs in situations involving anomalous routes in $\Cycle(k_1,k_2)$.
\end{example}

\subsection{The Coherence Lemma}\label{sec:coherence}

In this subsection, we prove a lemma establishing how the coherence diagram can be used to identify conflicting pairs of routes and the minimal face for a pair of routes.
This lemma will be our key tool for studying pulling triangulations.

Recall, as shown in Example~\ref{ex:anomalousbadbehavior}, that it is possible for an anomalous route $R$ to be both northwest and northeast of another anomalous route $S$.
Hence, the following lemma handles the anomalous route case separately.

\begin{lemma}[The Coherence Lemma]\label{lem:thecoherencelemma}
    Let $R$ and $S$ be routes in a locally anti-blocking framed DAG $G$ with $\bg$-polytope $P$.
    \begin{enumerate}[label=(\roman*)]

    \item If both $R$ and $S$ are anomalous:
    \begin{enumerate}
        \item $S\in \NW(R)\cup \SE(R)$ if and only if $R$ and $S$ are in conflict, and in this case the entire polytope $P$ is the minimal face containing both $R$ and $S$. 
        \item If $R$ and $S$ are coherent:
        \begin{enumerate}
            \item If one if $R$ or $S$ is exceptional, the minimal face containing them is $P$.
            \item Otherwise, the minimal face containing them is the edge from $R$ to $S$.
        \end{enumerate}
    \end{enumerate}
    \item If $R$ and $S$ are not both anomalous:
    \begin{enumerate}
        \item $S\in \NW(R) \cup \SE(R)$ if and only if $S$ is in conflict with $R$.
        Further, setting $R\squarette  S =: \{R, S, T_1, T_2\}$ as in Definition~\ref{def:diagramregions}, we have:
    \begin{enumerate}
        \item if one of $\{T_1,T_2\}$ is exceptional, then $R$ and $S$ lie on different facets of $P$.
        \item Otherwise, the convex hull of the routes in $R\squarette  S$ forms the minimal face of $P$ containing $R$ and $S$.
        \end{enumerate}     
        \item $S \in \NE(R) \cup \SW(R)$ if and only if $S$ is coherent with $R$ and the line segment from $R$ to $S$ is not an edge of $P$.
        \item $S$ lies outside of $\NW(R)\cup \NE(R)\cup \SE(R)\cup \SW(R)$ if and only if $S$ is coherent with $R$ and the line segment from $R$ to $S$ is an edge of $P$.
    \end{enumerate}
    \end{enumerate}
\end{lemma}

Before proving this lemma, we consider two examples.

\begin{example}\label{eg.caracol6}
    Figure~\ref{fig:caracol-tableau} shows the caracol graph $\Path(6)$ and its coherence diagram.
    Following our lab framing convention, the horizontal edges are framed with $2$'s and the arc edges are framed with $1$'s.
    This graph has eight exceptional routes, which are $(1^-,7^+)$ and $(i,i)$ for $i=1,\ldots, 7$.
    The boxes in the coherence diagram that contain dots correspond to the exceptional routes.
    The route $R=(2,5)$ has an open circle in its box.
    The routes corresponding to the green boxes are coherent with $R$ and share an edge with $R$ in the $\bg$-polytope $P$.
    The routes corresponding to the blue boxes are coherent with $R$, but do not share an edge in $P$.
    The routes corresponding to the red and yellow boxes are in conflict with $R$, and the routes corresponding to the yellow boxes do not lie on a common facet with $R$ in $P$.
\end{example}

\begin{figure}
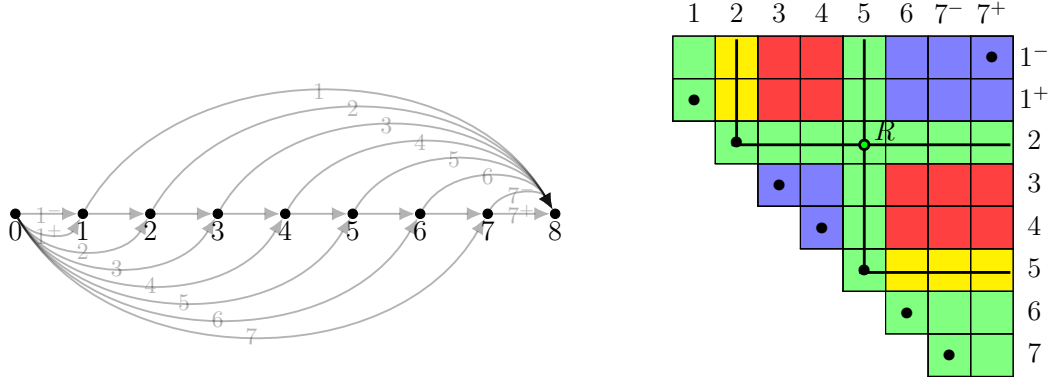

    \centering
    \includegraphics[height=1.7in]{figures/caracol6_graph.tex}
    \hspace{1cm}
    \includegraphics[height=2in]{figures/caracol_tableau.tex}
    \caption{
    On the left is the caracol graph whose inner subgraph is $\Path(6)$.
    On the right is the coherence diagram for this caracol graph.
    See Example~\ref{eg.caracol6} for an explanation of the color-coded boxes in the coherence diagram.
}
    \label{fig:caracol-tableau}
\end{figure}

\begin{example}\label{ex:anomalouslemma}
    Consider $\Cycle(3,2)$ with coherence diagram given in Figure~\ref{fig:cycle32diagram}.
    The anomalous routes are indicated by diamonds.
    The route $(1^-,4^-)$ is northwest of $(1^+,4^+)$, $(4^+,1^-)$, and $(4^+,1^+)$, and it is straightforward to check that these three routes are the only anomalous routes in conflict with $(1^-,4^-)$.

    The route $(1^+,4^-)$ is southeast of $(4^-,1^-)$, due to the identified edges on the diagram, and it is northwest of $(4^+,1^-)$ and $(4^+,1^+)$.
    It is straightforward to check that these are the only anomalous routes in conflict with $(1^+,4^-)$.

    The route $(1^+,4^+)$ is southeast of $(1^-,4^-)$, $(4^-,1^-)$, and $(4^+,1^-)$, due to the identified edges on the diagram.
    One can check that these are the only conflicting anomalous routes with $(1^+,4^+)$.

     Further, it is straightforward to check that the union of edges in any of these conflicting pairs contains an exceptional route, and hence by Corollary~\ref{cor:incoherence-square} the minimal face containing that pair is $P$.
\end{example}

\begin{proof}[Proof of Lemma~\ref{lem:thecoherencelemma}]

The proof of \textit{(i)} is a tedious case-by-case verification.
Note that there are only eight anomalous routes in any $\Cycle(k_1,k_2)$, and two of them are exceptional.
These cases are an extension of the partial analysis of pairs of anomalous routes given in Example~\ref{ex:anomalouslemma}.

We now consider the proof of \textit{(ii)}, which is the general case.
We begin by proving \textit{(ii)(a)}, which states $S\in \NW(R) \cup \SE(R)$ if and only if $S$ is in conflict with $R$.
Observe that once we prove \textit{(ii)(a)}, the claims in \textit{(ii)(a)(i)} and \textit{(ii)(a)(ii)} follow immediately from Corollary~\ref{cor:incoherence-square}.
For the forward implication, without loss of generality suppose that $S\in \NW(R)$.
The routes in $R\squarette  S$ correspond to the bounding corners of a rectangle within the coherence diagram.
We argue by cases as follows.
\begin{enumerate}
    \item Suppose three things: $R\squarette S$ is contained in the bounding strip, vertex $i$ is adjacent to the source, and $S$ contains $i^-$ and $R$ contains $i^+$ (the case where $i$ is adjacent to the sink is a similar argument).
    One of $i^+$ or $i^-$ might be denoted by $i$ in the case that $G=\Path(\bk)$ and $i$ is a leaf of the inner graph. 
    If the edges connecting $S$ and $R$ to the source are on different legs, then the routes conflict at $i$.
    Otherwise, since $R$ is to the right of $S$ in the coherent diagram, either $R$ exits from the leg containing $R\cap S$ on a different vertex from the one where $S$ exits or they both exit the leg at the final possible vertex.
    If they exit on different vertices and if $R\cap S$ lies on a $2$-leg, then $R$ exits the leg after $S$.
    If they exit on different vertices and if $R\cap S$ lies on a $1$-leg, then $S$ exits the leg after $R$.
    In both cases, $R$ and $S$ are in conflict.
    If $R$ and $S$ both enter and exit the leg at the source- and sink-adjacent vertices, respectively, then (recalling that $R$ and $S$ are not both anomalous) $R\squarette S$ must form a contiguous $2\times 2$ square in a corner of the bounding strip.
    Hence, $R\cap S$ is an entire leg and we have that $R$ and $S$ conflict across the leg.
    \item Suppose that $R\squarette S$ is not contained in the bounding strip, i.e., at least one element of $R\squarette S$ is contained in a staircase shape outside of the bounding strip, so that element corresponds to a route passing through the inner vertices of a fixed leg $L$ in $G$.
    Every element of $R\squarette S$ therefore is contained in the portion of the coherence diagram given by the routes that lie on $L$.
    Thus, the first and last inner vertices in the route $(R\cup S)\cap L$ are distinct, and since $S\in \NW(R)$, one of these is contained in $R$ and the other is contained in $S$.
    Since $R$ and $S$ lie on the same leg, and since the framing labels are constant for the edges on the leg, it must be that $R$ and $S$ conflict across $R\cap S$.
\end{enumerate}

We next prove the reverse implication for \textit{(ii)(a)}.
Suppose that $R$ and $S$ are in conflict, where $R$ and $S$ are not both anomalous.
The connected path $R\cap S$ must be contained within a leg of $G$, since the inner graph of $G$ is either a path or a cycle.
We consider the following cases.
\begin{enumerate}
    \item Suppose that for the conflicting routes $R$ and $S$, there exists a vertex $i$ with two edges connected to either the source or the sink such that one of the routes contains $i^+$ while the other contains $i^-$.
    Note that in the case where $G=\Path(\bk)$, one of these edges might be denoted by $i$ if the vertex $i$ is a leaf of the inner graph.
    We will assume that $i$ is adjacent to the source, as the argument is similar when $i$ is adjacent to the sink.
    Thus, $R$ and $S$ are both contained in the bounding strip for the diagram, within the rectangle indexed by rows $i^+$ and $i^-$.
    Further, $R$ and $S$ are in different rows and columns, since they must have distinct sink-adjacent edges, otherwise the routes would be identical to the right of vertex $i$.
    The other two routes contained in $R\cup S$ are precisely the routes given by the corners of the rectangle in the diagram defined by $R$ and $S$, as the route pairs for those two routes are obtained by exchanging the sink-adjacent edges in the route pairs for $R$ and $S$.
    Hence, either $R\in \NW(S)$ or $S\in \NW(R)$.
    \item Suppose that $R\cup S$ does not contain both $i^+$ and $i^-$ for any source-adjacent or sink-adjacent vertex $i$.
    In this case, the source-adjacent vertices for $R$ and $S$ are distinct, as are the sink-adjacent edges.
    Thus, the route pairs for $R$ and $S$ have distinct first coordinates and distinct second coordinates, and the other two routes in $R\cup S$ are obtained by exchanging the second coordinates in the route pairs for $R$ and $S$.
    This exchange of second coordinates produces routes corresponding to the $T_1$ and $T_2$ in $R\squarette S$, and hence either $R\in \NW(S)$ or $S\in \NW(R)$.
\end{enumerate}

We next prove cases \textit{(ii)(b)} and \textit{(ii)(c)}. 
Observe that due to case \textit{(ii)(a)}, it is immediate that $R$ and $S$ are coherent if and only if $S\notin \NW(R)\cup \SE(R)$.
This establishes the ``coherence'' condition in the if and only if statements for \textit{(ii)(b)} and \textit{(ii)(c)}.
We need to verify that $S\in \NE(R)\cup\SW(R)$ precisely when the line segment from $R$ to $S$ is not an edge of $P$.
Without loss of generality, suppose $S\in \NE(R)$.
By definition of $\NE(R)$, we have $S\in \NE(R)$ if and only if there exist conflicting routes $T$ and $U$ such that $T\squarette U=\{T,U,R,S\}$.
By part \textit{(ii)(a)} of Corollary~\ref{cor:incoherence-square}, it follows that the minimal face containing $T\squarette U$ is the quadrilateral face that is the convex hull of their $\bg$-vectors.
Because the quadrilateral is the smallest face containing $T$ and $U$, the line segment from $R$ to $S$ passes through the interior of this quadrilateral, hence is not an edge of $P$.

As a consequence of \textit{(ii)(a)} and \textit{(ii)(b)}, we have that routes outside of $\NW(R)\cup \NE(R)\cup \SE(R)\cup \SW(R)$ are precisely those that are coherent and for which the line segment in $P$ between them forms an edge of $P$.
This completes the proof of \textit{(ii)(c)}, and thus the proof of the lemma is complete.
\end{proof}

\subsection{Pulling triangulations and \texorpdfstring{$\bg$}{g}-polytopes}

Our goal in this subsection is to introduce the concept of pull-coherence and characterize those pulling sequences of the $\bg$-polytopes for $\Path(\bk)$ and $\Cycle(\bk)$, excluding $\Cycle(k_1,k_2)$, that yield DKK triangulations.
Throughout this subsection, we identify both a route and its $\bg$-vector by the same symbol.

\begin{definition}
    \label{def:pullcoherent}
    Let $P$ be $\bg$-polytope for $\Path(\bk)$ or $\Cycle(\bk)$, and let $\vec{0}$ denote the origin.
    A sequence of routes starting with the origin, denoted $V=(\vec{0},R_1,\ldots,R_i)$, is called \defn{pull-coherent} if it yields a pulling subdivision $\pull(P;V)$ that is refined by the DKK triangulation.
\end{definition}

\

The following lemmas form the core of our arguments regarding pulling triangulations that refine the DKK triangulation.

\begin{lemma}
    \label{lem:pullorigin}
    Let $P$ be the $\bg$-polytope for $\Path(\bk)$ or $\Cycle(\bk)$.
    The subdivision obtained by pulling $P$ at the origin is refined by the DKK triangulation of $P$.
\end{lemma}

\begin{proof}
    This follows from the fact that the DKK triangulation $\calT$ has the origin as a cone point, i.e., the origin is contained in every simplex in $\calT$.
\end{proof}

\begin{lemma}
    \label{lem:pullcoherencerefinement}
    Let $P$ be the $\bg$-polytope for $\Path(\bk)$ or $\Cycle(\bk)$, excluding $\Cycle(1,1)$.
    Let $\calT$ denote the DKK triangulation of $P$ and let $\calS=\pull(P;V)$ where $V=(\vec{0},R_1,\dots, R_i)$ is a pull-coherent sequence of routes.
    Let $R\notin V$ be a route.
    The triangulation $\calT$ is a refinement of $\pull(\calS;R)$ if and only if the following two conditions hold:
     \begin{enumerate}[label=\alph*)]
         \item $\west(R)\subseteq V$ or $\north(R)\subseteq V$
         \item $\east(R)\subseteq V$ or $\south(R)\subseteq V$
     \end{enumerate}   
\end{lemma}

\begin{proof}
    First, suppose $R$ is not anomalous.
    We prove the forward direction by contrapositive.
    Suppose that we have one of the following two conditions:
    \begin{itemize}
         \item[$a')$] $\west(R)\nsubseteq V$ and $\north(R)\nsubseteq V$
         \item[$b')$] $\east(R)\nsubseteq V$ and $\south(R)\nsubseteq V$
     \end{itemize}
     Without loss of generality, suppose $(a')$ holds.
     In this case, there exist a pair of routes $T_1\in W(R)\setminus V$ and $T_2\in N(R) \setminus V$ in the coherence diagram, which therefore yield a route $S\in NW(R)$ such that $R\squarette S=\{R,S,T_1,T_2\}$.
     Consider Lemma~\ref{lem:thecoherencelemma} \textit{(ii)(a)} applied to $R$ and $S$, implying that these routes are in conflict. Note that neither $T_1$ nor $T_2$ are exceptional (otherwise they would already be pulled), and so we are in case Lemma~\ref{lem:thecoherencelemma} \textit{(ii)(a)(ii)}. Therefore, the minimal face containing $R$ and $S$ is the quadrilateral convex hull of $R\squarette S$. We then cannot have $S\in V$, as this would have created the edge from $R$ to $S$, which would contradict that the DKK triangulation refines $\calS$ since $S$ is in conflict with $R$.
     But then pulling at $R$ would create the edge from $R$ to $S$, which means that $\calT$ is not a refinement of $\pull(\calS;R)$.

     For the reverse direction, suppose $S$ is a route that is in conflict with $R$.
     Suppose that $S\in \NW(R)$ (the case where $S\in \SE(R)$ is similar).
     Then there exists routes $T_1\in \west(R)$ and $T_2\in \north(R)$ such that $R\squarette S=\{R,S,T_1,T_2\}$.
     Without loss of generality, suppose $\west(R)\subseteq V$.
     Then $T_1$ has already been pulled in the sequence $V$, and by Theorem~\ref{thm:pulling1}, we have that $R$ and $S$ are separated.
     Thus, for any cell $Y$ in $\calS$ that contains $R$, every vertex of $Y$ is coherent with $R$. 
     Consider any maximal collection of coherent vertices of $Y$.
     This collection must contain $R$, since $R$ is coherent with each of them. 
     Hence, for any $T\in \calT$ with $T\subseteq Y$, the route $R$ is contained in $T$. 
     Thus, by Theorem~\ref{thm:iteratedpulling}, $\calT$ refines the subdivision obtained by pulling $\calS$ at $R$.

     Next, we consider the case where $R$ is anomalous. 
     The two anomalous exceptional routes correspond to the origin and hence are always in $V$. 
     We begin by showing the lemma applies when $R$ is one of the four anomalous routes that are adjacent (along an edge of the square) to an anomalous exceptional route in the coherence diagram. 
     In this case, $R$ satisfies the two conditions \textit{(a)} and \textit{(b)} of the lemma, and the forward direction follows immediately.
     For the reverse direction, we need to show that $\pull(\calS;R)$ is refined by $\calT$.
     Suppose $S$ is a route in conflict with $R$. 
     If $S$ is anomalous, it is already separated from $R$ by the initial pull at the origin, because in this case $R$ and $S$ lie on different facets by Lemma~\ref{lem:thecoherencelemma}.
     If $S$ is not anomalous, then $R\squarette S$ must contain an exceptional anomalous route, and by Lemma~\ref{lem:thecoherencelemma}~\textit{(ii)(a)(i)}, the vertices $R$ and $S$ lie on different facets and were separated by the initial pull at the origin.
     Hence, by the same argument as given in the non-anomalous case, we have that $\calT$ refines $\calS$.
    
    Finally, suppose that $R$ is anomalous but is not adjacent to an anomalous exceptional, i.e., $R$ is $(1^+, N^-)$ or $(N^+, 1^-)$ where our DAG is $\Cycle(k_1,k_2)$ and $N:=k_1+1$.
    Consider the case where $R= (1^+, N^-)$; the case of $(N^+, 1^-)$ follows from a similar argument. 
    
    We prove the forward direction by contrapositive. 
    Assume that we have a pair of routes $T_1\in \west(R)\setminus V$ and $T_2\in \north(R) \setminus V=\{(1^-,N^-)\}$, hence $T_2=(1^-,N^-)$; the case $T_1\in \east(R)\setminus V$ and $T_2\in \south(R) \setminus V$ follows from an identical argument. 
    This gives us a route $S$ with $R\squarette S = \{(1^+,N^-),S,T_1, (1^-,N^-)\}$. 
    If $S$ is not anomalous, then $T_1$ is not exceptional.
    Thus, the subdivision $\pull(\calS;R)$ is not refined by $\calT$, because $S$ could not have already been in $V$ (or else the edge from $R$ to $S$ would be in $\calS$, which contradicts that $\calT$ refines it) and thus if we pull at $R$, then the edge from $R$ to $S$ is included in the new subdivision, which is not refined by $\calT$.
    So, supposing instead that $S$ is anomalous, we have two cases: either $S=(N^-,1^-)$ or $S=(N^+,1^-)$.
    Suppose first that $S=(N^-,1^-)$, which implies that $T_1=(N^-,1^+)$.
    This is a contradiction, as $(N^-,1^+)$ is exceptional and hence already in $V$.
    Next suppose that $S=(N^+,1^-)$, which implies that $T_1=(N^+,1^+)$.
    In this case, $R$ and $S$ are separated in $\pull(P;V)$ by Lemma~\ref{lem:thecoherencelemma}~\textit{(i)(a)}.
    Observe that if any of the non-anomalous routes in the column indexed by $N^+$ are not in $V$, then pulling at $R$ introduces a conflicting edge and we are done.
    So, assume that all of the non-anomalous routes in the column indexed by $N^+$ are in $V$.
    Since $\max(k_1,k_2)>1$, there is at least one such route.
    This route is in conflict with $T_2=(1^-,N^-)$, which we assumed is not in $V$.
    However, this means that pulling $V$ did not yield a subdivision refined by $\calT$, and thus pulling $R$ similarly does not yield a subdivision refined by $\calT$.
    
    To prove the reverse implication in the case where $R= (1^+, N^-)$, consider any route $S$ in conflict with $R$. 
    If $S$ is anomalous, then $S$ is separated from $R$ by the initial pull at the origin. If $S$ is not anomalous, then our earlier argument in the non-anomalous case applies and shows that $\pull(\calS;R)$ is refined by $\calT$.
\end{proof}

\begin{definition}
    \label{def:dkkpullstep}
    Let $P$ be the $\bg$-polytope for $\Path(\bk)$ or $\Cycle(\bk)$, excluding $\Cycle(1,1)$.
    Let $V=(\vec{0},R_1,\dots, R_i)$ be a pull-coherent sequence for $P$.
    For a route $R\notin V$, we say the sequence $V=(\vec{0},R_1,\dots, R_i,R)$ is a \defn{DKK pull step} if it satisfies the condition given in Lemma~\ref{lem:pullcoherencerefinement}, i.e., if the following both hold.
    \begin{enumerate}[label=\alph*)]
         \item $\west(R)\subseteq V$ or $\north(R)\subseteq V$
         \item $\east(R)\subseteq V$ or $\south(R)\subseteq V$
     \end{enumerate}  
\end{definition}

The following corollary characterizes the pulling sequences that yield the DKK triangulation of $P$.

\begin{corollary}\label{cor:characterizing-pull-coherent-routes}
    Let $P$ be the $\bg$-polytope for $\Path(\bk)$ or $\Cycle(\bk)$, excluding $\Cycle(1,1)$.
    A pulling sequence for $P$ yields the DKK triangulation if and only if it consists entirely of DKK pull steps.
\end{corollary}

\begin{proof}
For the forward implication, we argue the contrapositive. 
If some route in the pulling sequence is not a DKK pull step, then by Lemma~\ref{lem:pullcoherencerefinement}, the resulting subdivision is not refined by the DKK triangulation.
For the reverse directions, if the pulling sequence consists entirely of DKK pull steps, then by Lemma~\ref{lem:pullcoherencerefinement}, it yields the DKK triangulation.
\end{proof}

We thus arrive at the following theorem.

\begin{theorem}
    \label{thm:dkkpulling}
    Let $P$ be the $\bg$-polytope for $\Path(\bk)$ or $\Cycle(\bk)$.
    The DKK triangulation of $P$ is a pulling triangulation.
\end{theorem}

\begin{proof}
    We first address $\Cycle(1,1)$. In this case, $P$ is a hexagon that is triangulated after pulling the origin. By Lemma~\ref{lem:pullorigin}, this is the DKK triangulation.
    
    Now, by Corollary~\ref{cor:characterizing-pull-coherent-routes}, any sequence of routes consisting entirely of DKK pull steps yields the DKK triangulation.
    An example of how to generate such a sequence is to start with the outer corners of the coherence diagram (corresponding to the long exceptional routes), and pull the pair of routes adjacent to each outer corner in sequence from upper-left to lower-right.
    Then, move to the set of diagonals formed by the boxes adjacent to the vertices that were pulled at this step, and pull those vertices along each diagonal moving from upper-left to lower-right.
    Iterating this, the final step in the process will be when a single diagonal remains, and pulling these vertices along the diagonal from upper-left to lower-right completes the pulling sequence.
\end{proof}

It follows that pulling orders on $P$ that produce the DKK triangulation are easily described by the coherence diagram.
They are precisely the sequences $(\vec{0},R_1,\dots, R_p)$ such that for all $i=1,\dots,p$, the routes $\{R_1, \dots, R_i\}$ form the union of the Ferrer's diagrams of a NE skew partition and a SW skew partition in the coherence diagram.
We can improve this analysis, by identifying which DKK pull steps do not alter the previous subdivision.

\begin{proposition}
    \label{prop:redundantDKKstep}
    Suppose $R$ yields a DKK pull step for the sequence $V=(\vec{0},R_1,\ldots,R_i)$.
    Then $\pull(P;V)=\pull(P;(\vec{0},R_1,\ldots,R_i,R))$ if and only if
    $\NE(R)\cup\SW(R) \subseteq V$.
    We call any DKK pull step satisfying this condition \defn{redundant}.
\end{proposition}

\begin{proof}
    For the forward direction, if the pulling subdivisions are equal, then there is an edge between $R$ and every route coherent with $R$, which implies that every route coherent with $R$ has been pulled in $V$.
    By Lemma~\ref{lem:thecoherencelemma}, this implies that $\NE(R)\cup\SW(R) \subseteq V$.

    For the converse direction, if $\NE(R)\cup\SW(R) \subseteq V$, then every route coherent with that does not already lie on an edge of $P$ has been pulled as part of the sequence $V$.
    Thus, every simplex in the DKK triangulation that contains $R$ already exists in $\pull(P,V)$, and hence the pulling triangulations are equal.
\end{proof}

\begin{figure*}
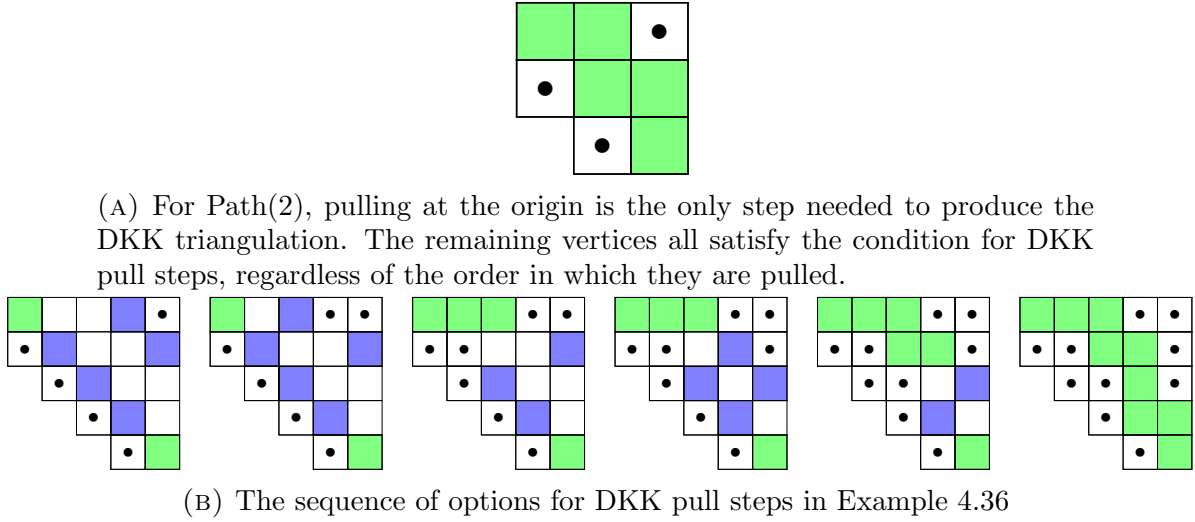

    \centering
    \begin{subfigure}[t]{0.8\textwidth}
        \centering
        \includegraphics[height=0.9in]{figures/caracol_pulling_2vxs.tex}
        \caption{For $\Path(2)$, pulling at the origin is the only step needed to produce the DKK triangulation.
        The remaining vertices all satisfy the condition for DKK pull steps, regardless of the order in which they are pulled.}
    \end{subfigure}
    \vspace{1em}
    \begin{subfigure}[t]{\textwidth}
        \centering
        \includegraphics[height=0.9in]{figures/caracol_pulling_4vxs.tex}
        \caption{The sequence of options for DKK pull steps in Example~\ref{ex:pathpulling}}.
    \end{subfigure}
    \caption{Diagrams representing pulling orders for $\Path(2)$ and $\Path(4)$.
    Bullets represent routes that have been pulled. 
    A box is shaded blue or green if the route can be pulled, where it is shaded green if it is a redundant DKK pull step.}
    \label{fig:dkkpullsteps}
\end{figure*}

\begin{example}
    \label{ex:pathpulling}
    Consider Figure~\ref{fig:dkkpullsteps}.
    In (A), the coherence diagram for $\Path(2)$ is shown.
    Here, pulling at the origin is the only step needed to produce the DKK triangulation, and all remaining routes are redundant.
    In (B), a sequence of irredundant DKK pull steps is shown.
    Once the diagram on the right-hand side is obtained, the subdivision is the DKK triangulation, and any subsequent DKK pull steps are redundant.
\end{example}

\section{Further directions}\label{sec:furtherdirections}

It was shown in~\cite{Polytopes2} that DKK triangulations of amply framed DAGs are closely related to the $\tau$-tilting theory~\cite{AIR} of certain gentle algebras. This connection was further extended in~\cite{Berggren}, by defining for a general gentle algebra $\Lambda$ a \emph{turbulence polyhedron} $\calF_1(\Lambda)$ which is triangulated by the $\tau$-tilting theory of $\Lambda$; for certain gentle algebras, this triangulated turbulence polyhedron is the same as the DKK-triangulated flow polytope of an associated amply framed DAG. The work~\cite{Berggren} moreover showed the existence of convex $\bg$-polytopes (as defined in~\cite{AHIKM}) for arbitrary representation-finite gentle algebras which restricts to the notion of $\bg$-polytope considered in this work. We expect that the main results of this paper can be generalized to the setting of representation-finite gentle algebras.

In particular, we expect that gentle algebras whose $\bg$-polytopes are locally anti-blocking are precisely the gentle Nakayama algebras, i.e., gentle algebras whose underlying quiver is either a directed path or an oriented cycle with some relations.
Furthermore, DKK triangulations of these $\bg$-polytopes should also be pulling triangulations where the pulling order can be read off from the Auslander-Reiten quiver of the gentle algebra, which in the gentle Nakayama case is equivalent to the coherence diagram (Remark~\ref{remk:ar-quiver}). Moreover, we believe that a similar theory of reading pulling orders for DKK-triangulated $\bg$-polytopes from the Auslander-Reiten quiver may exist for more general classes of $\bg$-polytopes (which are not necessarily locally anti-blocking), such as gentle algebras whose underlying quiver is a tree.

Finally, it would be interesting to ask whether the results about pulling can be extended from DKK triangulations of $\bg$-polytopes to that of their associated flow polytopes.    

\bibliographystyle{plain}
\bibliography{bibliography}

\end{document}